\pdfoutput=1
\documentclass[depthtwo,microtype]{nicestyle}

\author{\textsc{Ulrich Thiel}} 
\title{\textsc{Decomposition matrices are generically trivial}} 
\date{}

\usepackage[square,sort,comma,numbers]{natbib}

\begin{document}

\maketitle 
\fancyhead[RE]{\footnotesize\textsc{Decomposition matrices are generically trivial}}
\fancyhead[LO]{\footnotesize\textsc{Ulrich Thiel}}
\fancyhead[LE,RO]{\footnotesize\thepage}
\thispagestyle{empty}

\begin{abstract}
We establish a genericity property in the representation theory of a flat family of finite-dimensional algebras in the sense of Cline–Parshal–Scott. More precisely, we show that the decomposition matrices as introduced by Geck and Rouquier of an algebra which is free of finite dimension over a noetherian integral domain and which splits over the fraction field of this ring are generically trivial, i.e., they are trivial in an open neighborhood of the generic point of the spectrum of the base ring. This generalizes a classical result by Brauer in modular representation theory of finite groups. We furthermore show that this is true precisely on an open set in case all fibers of the algebra split. In this way we get a stratification of the base scheme such that decomposition maps are trivial on each stratum. Moreover, this defines a new discriminant of such algebras which generalizes Schur elements of simple modules for symmetric split semisimple algebras. We  provide some extensions to the theory of decomposition maps allowing us to work without the usual normality assumption on the base ring.
\end{abstract}

\blfootnote{Date: Jun 15, 2015 (first version: February 20, 2014). Final version to appear in Int. Math. Res. Not.} 
\blfootnote{\textsc{Ulrich Thiel}, Universität Stuttgart, Fachbereich Mathematik, Institut für Algebra und Zahlentheorie, Lehrstuhl für Algebra, Pfaffenwaldring 57, 70569 Stuttgart, Germany.} \blfootnote{Email: \texttt{thiel@mathematik.uni-stuttgart.de}}

\section*{Introduction}

In modular representation theory of finite groups the technique of \textit{$p$-modular reduction} is well-established and an important tool. Among others, it leads to the notion of \textit{decomposition matrices} and the \textit{Brauer–Cartan triangle}. This technique has been extended by Geck and Rouquier \cite{GR-Centers-Simple-Hecke} to a decent class of algebras over integral domains. It became an important tool for studying algebras involving parameters, so for example Hecke algebras (see \cite{GP-Coxeter-Hecke}, \cite{Geck.M;Jacon.N11Representations-of-H}, and \cite{Chlouveraki:2009aa}) and, more recently, rational Cherednik algebras (see \cite{Bonnafe.C;Rouquier.R13Cellules-de-Calogero}, \cite{Ginzburg:2003aa}, and \cite{Thiel:2015aa}). We list several further examples in \S\ref{examples}. In modular representation theory of finite groups it is a classical fact due to Brauer that the decomposition matrices of a fixed group are trivial for all but finitely many prime numbers (only for prime numbers dividing the group order they are non-trivial). We will show here that an appropriate version of this phenomenon still holds in the extended setting (see \S\ref{main_theorems}). In contrast to the classical fact, this result is of a more geometric nature.   

To explain this extension, let $A$ be an algebra which is free and of finite dimension over an integral domain $R$ with fraction field $K$ (this will be our usual setup). We emphasize the geometric point of view in which the prime ideals of $R$ can be considered as \textit{parameters} for algebras derived from $A$. Namely, we can \word{specialize} the algebra $A$ in a prime ideal $\fp$ of $R$ by passing to the algebra $A(\fp) =\rk(\fp) \otimes_R A \cong \rk(\fp) \otimes_{R/\fp} A/\fp A$, where $\rk(\fp) \dopgleich \mrm{Frac}(R/\fp)$ is the fraction field of $R/\fp$. In more geometric terms the fibers are the scalar extensions to the residue fields of the stalks of the locally free sheaf of algebras over $\Spec(R)$ defined by $A$. Specializing essentially means plugging in parameters into $A$ in a general sense (think for example of a Brauer, Hecke, or Cherednik algebra for a specific choice of parameters as being obtained from an appropriate generic version of these algebras via specialization, see \S\ref{examples} for examples illustrating this). The ultimate aim is to understand the representation theory (or some specific part of it) of all fibers of $A$. The correct approach to this problem—following the philosophy of Grothendieck and what Cline–Parshall–Scott \cite{Cline:1999aa} coined \textit{generic representation theory}—is to understand how the representation theory of $A(\fp)$ varies with $\fp$. To this end, it suffices to  understand how the representation theory of a special fiber $A(\fp)$ relates to the one of the only distinguished fiber of $A$, the \textit{generic fiber} $A(0) = A^K$, where $A^K = K \otimes_R A$. The hope is that this relation is generically strong so that once we understand the generic fiber $A^K$, we understand Zariski almost all fibers. 

Decomposition maps are the right tool to study how simple modules of the fibers vary with $\fp$, more precisely how they relate to the simple modules of the generic fiber. Geck and Rouquier constructed (under some assumptions on $A$ and $R$ we ignore in the introduction, see \S\ref{decomposition_morphisms}) for any prime ideal $\fp$ of $R$ a morphism $\rd_A^\fp:\rG_0(A^K) \rarr \rG_0(A(\fp))$ between Grothendieck groups, the so-called \textit{decomposition map} of $A$ in $\fp$. This map generalizes $\fp$-modular reduction of modules, and thus generalizes the classical $p$-modular reduction in modular representation theory of finite groups. In case $\rd_A^\fp$ is \textit{trivial} in the sense that it induces a bijection between the simple modules, the special fiber $A(\fp)$ has essentially the same simple modules as $A^K$—in particular the number of simple modules and their dimensions are the same. The central question we want to study here is in line with the philosophy explained above: are the decomposition maps for $A$ generically trivial, i.e., is $\rd_A^\fp$ trivial for all $\fp$ in an open neighborhood of the generic point $(0)$ of $\Spec(R)$? Our first main result (see \S\ref{main_theorems}) is that this is indeed true—provided the base ring $R$ is noetherian and $A$ has split generic fiber $A^K$. Our second main result (see \S\ref{main_theorems}) is that in case all fibers of $A$ are split, then the decomposition maps are trivial precisely on an open subset of $\Spec(R)$.

Generic triviality of decomposition maps was so far only known in the following two general cases:
\begin{enum_proof}
\item if $R$ is normal, noetherian, and one-dimensional, and the generic fiber of $A$ splits.
\item if $R$ is normal, and $A$ is symmetric and its generic fiber is split semisimple.
\end{enum_proof}
The first case is due to Geck \cite{Geck.M98Representations-of-H} who has given a general formulation of James's conjecture \cite{James:1990aa} using the notion of trivial decomposition maps and proves generic triviality of decomposition maps in this context under these assumptions to provide a heuristic supporting this conjecture. The second case uses the theory of Schur elements and follows from Tits's deformation theorem (see \cite{GP-Coxeter-Hecke}). These two cases are quite restrictive for certain applications, however. Restricted rational Cherednik algebras (see \cite{Gordon:2003aa}, \cite{Bonnafe.C;Rouquier.R13Cellules-de-Calogero}, and \cite{Thiel:2015aa}) for example are not covered by these results as in general neither the base rings are one-dimensional nor the generic fiber is semisimple (even though it is symmetric). Moreover, even results covered by the setting of Schur elements have the problem that we cannot apply them to restrictions $A/\fp A$ for prime ideals $\fp$ in general (what we would like to do to continue studying the fibers of $A$ on the locus where $\rd_A^\fp$ is non-trivial): on the one hand the generic fiber $A(\fp)$ of $A/\fp A$ will not be semisimple anymore if $\rd_A^\fp$ is non-trivial, and on the other hand the base ring $R/\fp$ of $A/\fp A$ may not be normal anymore. 

We resolve these difficulties here. Our proof of generic triviality is based on Geck's arguments in the one-dimensional setting accompanied by a theorem by Grothendieck on the existence of dominating discrete valuation rings. As we want to avoid the usual assumption on the normality of the base ring, we provide some extensions to the theory of decomposition maps. The effect of dropping this assumption is that there might be more than one decomposition map for a given prime—but we will show that if one of these is trivial, all of them are trivial and so the situation is indeed as nice as possible. The openness result (our second main theorem) is obvious in the one-dimensional setting but was not even known in the setting of Schur elements. Our proof is based on a recursive application of the genericity property and Grothendieck's notion of ind-constructible subsets of schemes. This allows us to stratify the base scheme $\Spec(R)$ into strata on which decomposition maps are trivial. In particular, we only have to study finitely many fibers of $A$ to understand all of them. Moreover, it allows us to attach a new form of a discriminant to  algebras encoding where decomposition maps are non-trivial. We show that this discriminant generalizes Schur elements of simple modules for symmetric split semisimple algebras. \\

The outline of this paper is as follows. In \S\ref{decomposition_morphisms} we review the theory of decomposition maps and discuss several extensions to the usual expositions in the literature with the aim to remove the normality assumption on the base ring. In \S\ref{main_theorems} we formulate our main results and illustrate them by some examples. Their proofs occupy the rest of the paper. In \S\ref{connection_with_jac} we show that triviality of decomposition maps is related to a certain behavior of the Jacobson radical under specialization, which we study in \S\ref{jacobson_behavior} preceded by a study of the split locus in \S\ref{split_locus}. Finally, in \S\ref{generic_representation_theory} we formalize the notion of properties of algebras and discuss some general tools which applied to our setting prove the openness result mentioned above. We finish with some open questions in \S\ref{questions}.

\begin{ack}
I am very thankful to Meinolf Geck and Gunter Malle for several comments on a preliminary version of this article. I am also thankful to Burkhard Külshammer for asking some enlightening questions after a talk I have given on this topic. Part of this work was supported by the \textit{DFG Schwerpunktprogramm ``Darstellungstheorie'' 1388} and by the \textit{DFG Schwerpunktprogramm ``Algorithmische und experimentelle Methoden in Algebra, Geometrie und Zahlentheorie'' 1489}.
\end{ack}

\tableofcontents

\section{Decomposition maps} \label{decomposition_morphisms}

Our central object of study is the decomposition map introduced by Geck and Rouquier \cite{GR-Centers-Simple-Hecke}. As indicated in the introduction, we aim to drop the usual assumption that the base ring is normal. To this end, we will need to extract and generalize several arguments in loc. cit. (and in the general exposition on decomposition maps by Geck and Pfeiffer \cite{GP-Coxeter-Hecke}). The main result of this section is Theorem \ref{all_decs_are_discrete} which states that for a \textit{noetherian} base ring \textit{any} decomposition map can be realized by a \textit{discrete} valuation ring. We argue that this follows from the fact that decomposition maps in a prime $\fp$ depend only on the primes lying over $\fp$ in the ``Brauer–Nesbitt attractor'', a certain ring contained in the normalization of the base ring. Combined with a result by Grothendieck on the existence of dominating discrete valuation rings this implies the theorem. 

\subsection{Specializations}
Let $R$ be an integral domain with quotient field $K$ and let $A$ be an $R$-algebra. We denote by $A^S = S \otimes_R A$ the scalar extension of $A$ to an $R$-algebra $S$ along the canonical morphism $R \rarr S$. For a prime ideal $\fp$ of $R$ we denote by $\rk(\fp)$ the residue field of $R$ in $\fp$, i.e., the fraction field of $R/\fp$, and we call
\[
A(\fp) \dopgleich \rk(\fp) \otimes_R A \cong \rk(\fp) \otimes_{R/\fp} A/\fp A \cong \rk(\fp) \otimes_{R_\fp} A_\fp
\]
the \word{specialization} (or \word{fiber}) of $A$ in $\fp$. This is a $\rk(\fp)$-algebra. In more geometric terms, $A(\fp)$ is the scalar extension to $\rk(\fp)$ of the stalk at $\fp$ of the quasi-coherent $\sO_X$-algebra on the affine scheme $X \dopgleich \Spec(R)$ defined by $A$. We call the specialization $A(0) = A^K$ in the generic point $(0)$ of $\Spec(R)$ the \word{generic fiber} of $A$. We denote by $\theta_A^\fp:A \rarr A(\fp)$ the canonical morphism. Moreover, we define $A|_\fp \dopgleich R/\fp \otimes_R A \cong A/\fp A$, which can be considered as the \word{restriction} of $A$ to the closed subscheme $\rV(\fp)$ of $\Spec(R)$, where $\rV(\fp)$ denotes the closure of $\fp$ in $\Spec(R)$. 

In \S\ref{examples} we discuss several explicit examples illustrating that specialization essentially means plugging in parameters for an algebra whose definition involves parameters (like Brauer, Hecke, or rational Cherednik algebras). 
For most of our results (and for the existence of decomposition maps already) we need to assume that $A$ is an $R$-algebra which is free and of finite dimension as an $R$-module. We will simply say that $A$ is \word{finite free} in this case. The specialization $A(\fp)$ is then of the same dimension over $\rk(\fp)$ as $A$ is over $R$ and $A|_\fp$ is a free $(R/\fp)$-algebra, also of the same dimension.

\begin{leftbar}
\mbox{}\vspace{8pt}

\noindent Throughout, we assume that $R$ is an integral domain with fraction field $K$ and that $A$ is finitely generated and free as an $R$-module (we say that $A$ is finite free for short). \\
\end{leftbar}

\subsection{Gates}

Decomposition maps yield a connection between the generic fiber $A^K$ and a special fiber $A(\fp)$. This connection is set up by an an auxiliary ring in $K$ dominating $R_\fp$. We formalize this in the following definition which we prefix by an elementary lemma. Throughout, we denote by $\rG_0(A(\fp))$ the \word{Grothendieck group} of $A(\fp)$, i.e., the zeroth $\rK$-group of the category of finite-dimensional $A(\fp)$-modules. Recall that this is the free abelian group with basis given by the classes of simple $A(\fp)$-modules.

\begin{lemma} \label{d_field_ext_injective}
Let $A$ be a finite-dimensional algebra over a field $K$ and let $\theta:K \hookrightarrow L$ be an extension field of $K$. Then the induced morphism $\rd_A^\theta:\rG_0(A) \rarr \rG_0(A^L)$ is injective and we can thus view $\rG_0(A)$ as a subgroup of $\rG_0(A^L)$.
\end{lemma}

\begin{proof}
Let $(S_i)_{i \in I}$ be a system of representatives of the simple $A$-modules. For each $i \in I$ let $(T_{ij})_{j \in J_i}$ be a system of representatives of the simple constituents of the $A^L$-module $S_i^L$. Then $(T_{ij})_{i \in I, j \in J_i}$ is a system of representatives of the simple $A^L$-modules by \cite[7.13]{Lam-Noncommutative}. Hence, the matrix $\rD_A^\theta$ of the morphism $\rd_A^\theta$ in suitably sorted natural bases is in column echelon form, has no zero columns and has no zero rows. In particular, $\rd_A^\theta$ is injective.
\end{proof}

\begin{definition} \label{gate_definition}
If $\fp$ is a prime ideal of $R$, let us call a \word{weak $A$-gate} in $\fp$ any local ring $\sO$ between $R$ and $K$ with maximal ideal $\fm$ such that:
\begin{enum_thm}
\item $R \cap \fm = \fp$,
\item any finite-dimensional $A^K$-module $V$ has an $\sO$-free $A^\sO$-form, i.e., an $\sO$-free $A^\sO$-module $\wt{V}$ such that $A^K \otimes \wt{V} \cong V$ as $A$-modules via the canonical morphism.
\end{enum_thm}
We define an \word{$A$-gate} in $\fp$ to be a weak $A$-gate in $\fp$ satisfying the following additional property:
\begin{enum_thm}[resume]
\item \label{gate_definition_G} the reductions $\wt{V}/\fm \wt{V}$ of finitely generated $\sO$-free $A^\sO$-modules $\wt{V}$ are contained in the subgroup $\rG_0(A(\fp))$ of $\rG_0(A^\sO(\fm))$.
\end{enum_thm}
\end{definition}

An $A$-gate is precisely designed for producing a morphism $\rG_0(A^K) \rarr \rG_0(A(\fp))$ generalizing reduction modulo $\fp$. Before we discuss this let us first record a few observations. What is immediately clear from the definitions is the following.

\begin{lemma}
If $\sO$ is a weak $A$-gate in $\fp$, then any local ring in $K$ dominating $\sO$ is also a weak $A$-gate in $\fp$. 
\end{lemma}

We will see in Lemma \ref{dominating_gates} that the same also holds for $A$-gates but this is not trivial any more. It is a standard commutative algebra fact (see for example \cite[5.1]{Gol-Character-Theory}) that there always exists a valuation ring in $K$ dominating $R_\fp$. This follows essentially from the fact that valuation rings are precisely maximal elements in the set of local subrings of a field (see \cite[VI.1.2]{Bou-Commutative-Algebra-1-7}). It is a further standard fact (see \cite[7.3.7]{GP-Coxeter-Hecke}) that such a valuation ring satisfies the second property in Definition \ref{gate_definition}. This follows essentially from the fact that finitely generated torsion free modules over a valuation ring are already free (see \cite[5.2]{Gol-Character-Theory}). This proves the following lemma.

\begin{lemma}
In any $\fp$ there is a weak $A$-gate. In fact, any valuation ring in $K$ dominating $R_\fp$ is a weak $A$-gate. 
\end{lemma}

The last condition in Definition \ref{gate_definition} holds for example independently of $\sO$ if $A(\fp)$ splits. Recall that a finite-dimensional algebra $A$ over a field $K$ splits if and only if one (all) of the following equivalent conditions holds (see \cite{Lam-Noncommutative}):
\begin{enum_thm}
\item All simple $A$-modules remain simple under arbitrary field extensions of $K$.
\item $A/\Jac(A)$ is a finite direct product of matrix algebras over $K$, where $\Jac(A)$ is the Jacobson radical of $A$.
\item $\End_{A}(S) = K$ for any simple $A$-module $S$.
\item The natural map $A \rarr \End_{K}(S)$ is surjective for any simple $A$-module $S$.
\end{enum_thm}
Of course, if $\rk(\fp)$ is algebraically closed, then $A(\fp)$ splits. We record our observations again.

\begin{lemma} \label{split_gates_exist}
If $A(\fp)$ splits, there is an $A$-gate in $\fp$. In fact, any weak $A$-gate in $\fp$ (for example a valuation ring in $K$ dominating $R_\fp$) is already an $A$-gate in $\fp$.
\end{lemma}

Although splitting of $A(\fp)$ is our usual assumption for ensuring the existence of an $A$-gate later, note that if $R$ is a Prüfer domain—this is a Dedekind domain if and only if it is noetherian—then the localization $R_\fp$ is a valuation ring and thus indeed an $A$-gate in $\fp$ without having to assume that $A(\fp)$ splits. As our aim is to investigate the case of base rings of arbitrary dimension, we will essentially not come across this case.

\subsection{The Brauer–Nesbitt map}

Before we come to decomposition maps we recall an important ingredient in proving their existence and uniqueness—the Brauer–Nesbitt map. The reason we address this here is that we can give a proof of its unconditional injectivity, generalizing \cite[Proposition 2.5]{GR-Centers-Simple-Hecke} and \cite[Lemma 7.3.2]{GP-Coxeter-Hecke}. In this way we can drop this assumption and make the theory a bit slicker. Moreover, it plays an important role in understanding decomposition maps realized by different $A$-gates. This is needed for the proof of Theorem \ref{all_decs_are_discrete}. \\

First, we recall the definition of the Brauer–Nesbitt map. Throughout, we denote by $\rG_0^+$ the subsemigroup of the Grothendieck group $\rG_0$ formed by the classes of honest (i.e., non-virtual) modules. Let $\fp$ be a prime ideal of $R$ and let $L_{\fp }$ be an extension field of $\rk(\fp )$. Then the map \label{symb_bnA}
\[
\begin{array}{rcl}
 \mrm{bn}_{A}^{\fp ,L_{\fp }}:\mrm{G}_0^+(A(\fp )^{L_{\fp }}) & \longrightarrow & \Hom_{\cat{Set}}(A, L_{\fp } \lbrack X \rbrack) \\
 \lbrack V \rbrack & \longmapsto & (a \mapsto \chi_{\rho_V(\ol{a})}) \;,
\end{array}
\]
where $\chi_{\rho_V(\ol{a})} \in K \lbrack X \rbrack$ denotes the characteristic polynomial of the $L_{\fp }$-endomorphism $\rho_V(\ol{a})$ of $V$ defined by the action of the image $\ol{a}$ of $a\in A$ in $A/\fp A \subs A(\fp ) \subs A(\fp )^{L_{\fp }}$ on $V$, is a well-defined semigroup morphism. 
The map $\mrm{bn}_A^{\fp ,L_{\fp }}$ is called the \word{Brauer--Nesbitt map} of $A$ in $(\fp ,L_{\fp })$. We simply write $\mrm{bn}_A^{\fp }$ for $\mrm{bn}_A^{\fp ,\rk(\fp )}$. 

\begin{definition} \label{symb_tMtheta}
If $M$ is a set and $\theta:R \rarr S$ is a morphism of commutative rings, we denote by $\rt_M^\theta$ the canonical morphism $\Hom_{\cat{Set}}(M, R\lbrack X \rbrack) \rarr \Hom_{\cat{Set}}(M,S \lbrack X \rbrack)$ induced by $\theta$.
\end{definition}

Note that if $\theta$ is injective, then so is $\rt_M^\theta$. The following lemma is slightly more general than \cite[7.3.4]{GP-Coxeter-Hecke} but follows in the same way by writing the action of algebra elements explicitly as matrices.

\begin{lemma} \label{brauer_nesbitt_field_ext}
Let $L_{\fp }$ be an extension field of $\rk(\fp )$ and let $\theta:L_{\fp } \hookrightarrow L_{\fp }'$ be a field extension. Then the diagram
\[
\begin{tikzcd}
\rG_0^+(A(\fp )^{L_{\fp }}) \arrow{rr}{\mrm{bn}_A^{\fp ,L_{\fp }}} \arrow{d}[swap]{\rd_{A(\fp )^{L_{\fp }}}^\theta} && \Hom_{\cat{Set}}(A, L_{\fp } \lbrack X \rbrack) \arrow{d}{\rt_A^{\theta}} \\
\rG_0^+(A(\fp )^{L_{\fp }'}) \arrow{rr}[swap]{\mrm{bn}_A^{\fp ,L_{\fp }'}} && \Hom_{\cat{Set}}(A, L_{\fp }' \lbrack X \rbrack) 
\end{tikzcd}
\]
commutes.
\end{lemma}

\begin{proposition} \label{brauer_nesbitt}
The Brauer--Nesbitt map $\mrm{bn}_{A}^{\fp ,L_{\fp }}$ is injective for any prime ideal $\fp $ of $R$ and any extension field $L_{\fp }$ of $\rk(\fp )$.
\end{proposition}

\begin{proof} 
Let $L_{\fp }'$ be a splitting field of $A(\fp ) = (A/\fp A)^{\rk(\fp )}$ containing $L_{\fp }$. If $\theta:L_{\fp } \hookrightarrow L_{\fp }'$ denotes the embedding, then it follows from Lemma \ref{brauer_nesbitt_field_ext} applied to the $(R/\fp )$-algebra $A/\fp A$, the prime ideal $(0) \in \Spec(R/\fp )$ and the field extension $\theta$ that the diagram
\[
\begin{tikzcd}
\rG_0^+(A(\fp )^{L_{\fp }}) \arrow{rr}{\mrm{bn}_{A/\fp A}^{(0),L_{\fp }}} \arrow{d}[swap]{\rd_{A(\fp )^{L_{\fp }}}^\theta} && \Hom_{\cat{Set}}(A/\fp A, L_{\fp } \lbrack X \rbrack) \arrow{d}{\rt_{A/\fp A}^{\theta}} \\
\rG_0^+(A(\fp )^{L_{\fp }'}) \arrow{rr}[swap]{\mrm{bn}_{A/\fp A}^{(0),L_{\fp }'}} && \Hom_{\cat{Set}}(A/\fp A, L_{\fp }' \lbrack X \rbrack) 
\end{tikzcd}
\]
commutes. An application of the Frobenius--Schur theorem \cite[3.41]{CR-Methods-1} implies that the family of irreducible $A(\fp )^{L_{\fp }'}$-characters is linearly independent in the $L_{\fp }'$-module of class functions on $A(\fp )^{L_{\fp }'}$. Hence, the lemma \cite[7.3.2]{GP-Coxeter-Hecke} applied to the $(R/\fp )$-algebra $A/\fp A$ and the field $L_{\fp }'$ shows that $\mrm{bn}_{A/\fp A}^{(0),L_{\fp }'}$ is injective. Since $\rd_{A(\fp )^{L_{\fp }}}^\theta$ is injective by Lemma \ref{d_field_ext_injective} and since $\rt_{A/\fp A}^\theta$ is obviously injective, the commutativity of the above diagram thus implies that $\mrm{bn}_{A/\fp A}^{(0),L_{\fp }}$ is injective. 
The map $\mrm{bn}_A^{\fp ,L_{\fp }}$ we are interested in is equal to $(q_A^{\fp })^* \circ \mrm{bn}_{A/\fp A}^{(0),L_{\fp }}$, where 
\[
(q_A^{\fp })^*: \Hom_{\cat{Set}}(A/\fp A, L_{\fp })\lbrack X \rbrack) \rarr \Hom_{\cat{Set}}(A, L_{\fp }\lbrack X \rbrack)
\]
is the map induced by composing maps with the quotient morphism $q_A^{\fp }:A \rarr A/\fp A$. 
Now, if $\mrm{bn}_A^{\fp ,L_{\fp }}(\lbrack V \rbrack) = \mrm{bn}_A^{\fp ,L_{\fp }}(\lbrack W \rbrack)$ for some $A(\fp )$-modules $V$ and $W$, then by definition
\begin{align*}
\mrm{bn}_{A/\fp A}^{(0),L_{\fp }}(\lbrack V \rbrack) \circ q_A^{\fp } &= ((q_A^{\fp })^* \circ \mrm{bn}_{A/\fp A}^{(0),L_{\fp }})( \lbrack V \rbrack) = \mrm{bn}_A^{\fp ,L_{\fp }}(\lbrack V \rbrack) \\ &= \mrm{bn}_A^{\fp ,L_{\fp }}(\lbrack W \rbrack) = ((q_A^{\fp })^* \circ \mrm{bn}_{A/\fp A}^{(0),L_{\fp }})( \lbrack W \rbrack) = \mrm{bn}_{A/\fp A}^{(0),L_{\fp }}(\lbrack W \rbrack) \circ q_A^{\fp } \; .
\end{align*}
Using the fact that $q_A^{\fp }$ is surjective and thus right cancelable, we conclude that $\mrm{bn}_{A/\fp A}^{(0),L_{\fp }}(\lbrack V \rbrack)  = \mrm{bn}_{A/\fp A}^{(0),L_{\fp }}(\lbrack W \rbrack)$ and the above now implies that $\lbrack V \rbrack = \lbrack W \rbrack$, i.e., $\mrm{bn}_A^{\fp ,L_{\fp }}$ is injective.
\end{proof}

\subsection{Decomposition maps}

The injectivity of the Brauer–Nesbitt map is the central ingredient in proving that decomposition maps are well defined. This fact is due to Geck and Rouquier \cite{GR-Centers-Simple-Hecke}. In this section we essentially extract and generalize their arguments to show that decomposition maps exists for any choice of $A$-gate (see Corollary \ref{dec_morphisms_exist}). The additional details we give here enable us to study their dependence on the choice of the $A$-gate in the next section.

\begin{definition}
The \word{Brauer–Nesbitt attractor} $\Omega_A$ of $A$ is the intersection of all rings $\Omega$ between $R$ and $K$ with the property that the image of $\mrm{bn}_A^{(0)}$ is contained in the subset $\Hom_{\cat{Set}}(A, \Omega \lbrack X \rbrack)$ of $\Hom_{\cat{Set}}(A, K \lbrack X \rbrack)$. This ring is the unique minimal one with this property. 
\end{definition}

\begin{lemma} \label{bn_attractor_finite_type}
$\Omega_A$ is integral over $R$ and an $R$-algebra of finite type.
\end{lemma}

\begin{proof}
Let $(a_i)_{i=1}^n$ be an $R$-basis of $A$ and let $(S_j)_{j=1}^m$ be a system of simple $A^K$-modules. Let $C$ be the set of coefficients of the characteristic polynomials $\chi_{\rho_{S_j}}(a_i)$, where $1 \leq i \leq n$ and $1 \leq j \leq m$. This set is clearly finite. We claim that $\Omega_A$ is the $R$-subalgebra of $K$ generated by $C$. Clearly, $R \lbrack C \rbrack \subs \Omega_A$. To prove equality it is (due to the additivity of the Brauer–Nesbitt map) enough to show that the coefficients of $\mrm{bn}_A^{(0)}(\lbrack S_j \rbrack)(a) = \chi_{\rho_{S_j}}(a)$ are contained in $R \lbrack C \rbrack$ for all $a \in A$. It is an elementary fact that the coefficients of the characteristic polynomial of a sum of two matrices are polynomials in the coefficients of the characteristic polynomials of the two matrices (see also \cite{Amitsur:1979aa}). Hence, if $a$ is an $R$-linear combination of the $a_i$, then the coefficients of $\chi_{\rho_{S_j}}(a)$ are contained in $R \lbrack C \rbrack$. As $(a_i)_{i=1}^n$ is an $R$-basis of $A$, the claim follows. That $\Omega_A$ is integral over $R$ follows from a standard fact about integrality of the coefficients of characteristic polynomials (see \cite[Theorem 7.3.8]{GP-Coxeter-Hecke}).
\end{proof}

This implies in particular that the extension $R \subs \Omega_A$ is integral and so we have a surjective morphism between their spectra. 
As the normalization of an integral domain $R$ in its field of fractions $K$ is the intersection of all valuation rings between $R$ and $K$, we immediately obtain the following lemma.

\begin{lemma} \label{gates_containing_omega}
For any weak $A$-gate $\sO$ in $\fp$ there is a weak $A$-gate $\sO'$ in $\fp$ dominating $\sO$ and containing $\Omega_A$. In fact, any valuation ring in $K$ dominating $\sO$ satisfies this.
\end{lemma}

The following proposition can essentially be extracted from the proof of \cite[Theorem 7.4.3]{GP-Coxeter-Hecke}. We state it in a more general form here and give the full proof for convenience.

\begin{proposition} \label{bn_reduction_theorem}
Let $(\sO,\fm)$ be a weak $A$-gate in $\fp$ containing $\Omega_A$. Then the relation
\[
 \mrm{bn}_A^{\fp ,\rk(\fm )} ( \lbrack \wt{V}/\fm \wt{V} \rbrack) = \rt_A^{\theta^{\fm }} \circ \mrm{bn}_A^{(0)}( \lbrack \wt{V}^K \rbrack) 
\]
holds for any finitely generated $\sO$-free $A^\sO$-module $\wt{V}$. Here, $\theta^\fm:\sO \rarr \rk(\fm)$ is the canonical morphism.
\end{proposition}

\begin{proof}
First note that due to the assumption $\Omega_A \subs \sO$, the right hand side of the equation is well-defined. Let $\sB$ be an $\sO$-basis of $\wt{V}$ and let $a \in A$. Let $(m_{ij}) \in \mrm{Mat}_{n}(\sO)$ with $n \dopgleich \dim_{\sO}(V)$ be the matrix describing the action of $a \otimes 1 \in A^{\sO}$ on $\wt{V}$ with respect to $\sB$. Then $(\theta^{\fm }(m_{ij})) \in \mrm{Mat}_n(\rk(\fm ))$ is the matrix describing the action of $a \otimes 1 \in A^{\rk(\fm )}$ on $\wt{V}^{\rk(\fm )} = \wt{V}/\fm \wt{V}$ with respect to $\sB^{\rk(\fm )}$ and it follows that $\mrm{bn}_A^{\fp ,\rk(\fm )}(\lbrack \wt{V}/\fm \wt{V} \rbrack)(a \otimes 1)$ is the characteristic polynomial of this matrix. 
On the other hand, since $(m_{ij})$ is the matrix describing the action of $a \otimes 1 \in A^K$ on $\wt{V}^K$ with respect to $\sB^K$, it follows that $\rt_A^{\theta^{\fm }} \circ \mrm{bn}_A^{(0)}( \lbrack \wt{V}^K \rbrack)(a \otimes 1)$ is computed by first computing the characteristic polynomial of $(m_{ij})$, which has coefficients in $\sO$ since $\Omega_A \subs \sO$, and then reducing it modulo $\fm $, while $\mrm{bn}_A^{\fp ,\rk(\fm )}(\lbrack \wt{V}^{\rk(\fm )} \rbrack)(a \otimes 1)$ is computed by first reducing $(m_{ij})$ modulo $\fm $ and then computing the characteristic polynomial. As the operations of reducing and computing the characteristic polynomial commute and since all maps are $\rk(\fm )$-linear, the equality follows.
\end{proof}

To be precise in the following we denote for a local ring $(\sO,\fm)$ dominating $R_\fp$ by $\gamma_A^{\fp,\fm}$ the morphism $\rG_0(A(\fp)) \hookrightarrow \rG_0(A^\sO(\fm))$ induced by $\rk(\fp) \hookrightarrow \rk(\fm)$. It is injective by Lemma \ref{d_field_ext_injective} and thus an isomorphism onto its image.

\begin{proposition} \label{reduction_well_defined}
Let $(\sO,\fm)$ be a weak $A$-gate in $\fp$. Let $\wt{V}$ and $\wt{W}$ be two $\sO$-free $A^\sO$-forms of a finite-dimensional $A^K$-module $V$. Then
\[
\lbrack \wt{V}/\fm\wt{V} \rbrack = \lbrack \wt{W}/\fm\wt{W} \rbrack 
\]
in $\rG_0(A^\sO(\fm))$.
\end{proposition}

\begin{proof}
First, let us assume that $\sO$ contains $\Omega_A$. We have $\wt{V}^K \cong V \cong \wt{W}^K$ and so by Proposition \ref{bn_reduction_theorem} we get
\[
\mrm{bn}_A^{\fp ,\rk(\fm )} ( \lbrack \wt{V}/\fm \wt{V} \rbrack) = \rt_A^{\theta^{\fm }} \circ \mrm{bn}_A^{(0)}( \lbrack \wt{V}^K \rbrack)  = \rt_A^{\theta^{\fm }} \circ \mrm{bn}_A^{(0)}( \lbrack \wt{W}^K \rbrack)  = \mrm{bn}_A^{\fp ,\rk(\fm )} ( \lbrack \wt{W}/\fm \wt{W} \rbrack) \;.
\]
The Brauer–Nesbitt map is injective by Proposition \ref{brauer_nesbitt}, proving the claim in this case.
Now, assume that $\sO$ is arbitrary. Let $(\sO',\fm')$ be a weak $A$-gate in $\fp$ dominating $\sO$ and containing $\Omega_A$. This exists by Lemma \ref{gates_containing_omega}. Both $\wt{V}^{\sO'}$ and $\wt{W}^{\sO'}$ are two $\sO'$-free $A^{\sO'}$-forms of $V$. By the aforementioned we have 
\begin{align*}
\gamma_{A^\sO}^{\fm,\fm'}(\lbrack \wt{V}/\fm \wt{V} \rbrack ) & = \lbrack (\wt{V}/\fm\wt{V})^{\rk(\fm')} \rbrack = \lbrack \wt{V}^{\sO'}/\fm' \wt{V}^{\sO'} \rbrack = \lbrack \wt{W}^{\sO'}/\fm' \wt{W}^{\sO'} \rbrack \\ & = \lbrack (\wt{W}/\fm\wt{W})^{\rk(\fm')} \rbrack = \gamma_{A^\sO}^{\fm,\fm'}(\lbrack \wt{W}/\fm \wt{W} \rbrack ) \;.
\end{align*}
Since $\gamma_{A^\sO}^{\fm,\fm'}$ is injective, the claim follows.
\end{proof}

The preceding proposition immediately implies the following result about the existence of decomposition maps for any choice of $A$-gate.

\begin{corollary} \label{dec_morphisms_exist}
If $(\sO,\fm)$ is an $A$-gate in $\fp$, then there is a unique morphism 
\[
\rd_A^{\fp,\sO}: \rG_0(A^K) \rarr \rG_0(A(\fp))
\]
of Grothendieck groups such that for any finitely generated $A^K$-module $V$ and any $\sO$-free $A^\sO$-form $\wt{V}$ of $V$ we have
\[
\rd_A^{\fp,\sO}( \lbrack V \rbrack ) = (\gamma_A^{\fp,\fm})^{-1} ( \lbrack \wt{V}/\fm\wt{V} \rbrack ) \;.
\]
\end{corollary}

Here, we used the fact that $\lbrack \wt{V}/\fm\wt{V} \rbrack$ lies in the image of $\gamma_A^{\fp,\fm}$ by definition of an $A$-gate. The map $\rd_A^{\fp,\sO}$ is called the \word{decomposition map} of $A$ in $\fp$ with respect to $\sO$.  This map is essentially just reduction of modules modulo $\fp$, except that we might have to pass through the extension ring $(\sO,\fm)$ to find $\sO$-free forms and then reduce modulo $\fm$. 

\subsection{Dependence on the choice of $A$-gates}

An intricate problem is to understand how decomposition maps depend on the choice of the $A$-gate used for their definition. Geck and Rouquier used the Brauer–Nesbitt map to show that in case $R$ is normal all decomposition maps for a prime $\fp$ coincide. We extend this idea here to show that without assuming that $R$ is normal the decomposition maps in $\fp$ only depend on primes lying over $\fp$ in the Brauer–Nesbitt attractor of $A$. 

\begin{proposition} \label{dominating_gates}
Let $\sO$ be an $A$-gate in $\fp$. Then any local ring $\sO'$ in $K$ dominating $\sO$ is also an $A$-gate in $\fp$ and $\rd_A^{\fp,\sO} = \rd_A^{\fp,\sO'}$.
\end{proposition}

\begin{proof}
We already know that $\sO'$ is a weak $A$-gate. Let $\fm$ be the maximal ideal of $\sO$ and let $\fm'$ be the maximal ideal of $\sO'$. Let $\wt{V}$ be a finitely-generated $\sO'$-free $A^{\sO'}$-module and let $V$ be the scalar extension of $\wt{V}$ to $A^K$. Let $\wt{W}$ be an $\sO$-free $A^{\sO}$-form of $V$. Then $\wt{W}^{\sO'}$ is another $\sO'$-free $A^{\sO'}$-form of $V$ and so it follows from Proposition \ref{reduction_well_defined} that 
\[
\gamma_{A^\sO}^{\fm,\fm'}(\lbrack \wt{W}/\fm\wt{W} \rbrack) = \lbrack (\wt{W}/\fm \wt{W})^{\rk(\fm')} \rbrack = \lbrack \wt{W}^{\sO'}/\fm' \wt{W}^{\sO'} \rbrack = \lbrack \wt{V}/\fm' \wt{V} \rbrack 
\]
in $\rG_0(A^{\sO'}(\fm'))$. Since $\sO$ is an $A$-gate, the reduction $\lbrack \wt{W}/\fm \wt{W} \rbrack$ is contained in the image of $\rG_0(A(\fp))$ in $\rG_0(A^\sO(\fm))$. The above equation thus shows that the reduction $\lbrack \wt{V}/\fm' \wt{V} \rbrack$ is contained in the image of $\rG_0(A(\fp))$ in $\rG_0(A^{\sO'}(\fm'))$. Hence, $\sO'$ is an $A$-gate. From the above equation we get
\[
(\gamma_A^{\fp,\fm})^{-1}(\lbrack \wt{W}/\fm\wt{W} \rbrack) = (\gamma_{A}^{\fp,\fm'})^{-1}(\lbrack \wt{V}/\fm'\wt{V} \rbrack)
\]
and so
\[
\rd_A^{\fp,\sO}( \lbrack V \rbrack) = (\gamma_A^{\fp,\fm})^{-1}( \lbrack \wt{W}/\fm\wt{W} \rbrack ) = (\gamma_A^{\fp,\fm'})^{-1}( \lbrack \wt{V}/\fm'\wt{V} \rbrack) = \rd_A^{\fp,\sO'}( \lbrack V \rbrack) \;.
\]
\end{proof}

The lemma above shows that any decomposition map is already realized by a valuation ring so that we have at most one decomposition map per valuation ring in $K$ dominating $R_\fp$. We will improve this result a bit further. 

\begin{proposition} \label{brauer_nesbitt_gate}
If $(\sO,\fm)$ is an $A$-gate in $\fp$, then the diagram
\[
\begin{tikzcd}
\rG_0^+(A^K) \arrow[hookrightarrow]{rrr}{\mrm{bn}_A^{(0)}} \arrow{d}[swap]{\rd_A^{\fp,\sO}} & & & \Hom_{\cat{Set}}(A, \Omega_A \lbrack X \rbrack) \arrow{d}{\rt_A^{\theta^{\fm  \cap \Omega_A}}} \\
\rG_0^+(A(\fp )) \arrow[hookrightarrow]{rrr}[swap]{\rt_A^{\iota^{\fp ,\fm  \cap \Omega_A}} \circ \mrm{bn}_A^{\fp }} & & & \Hom_{\cat{Set}}(A, \rk(\fm  \cap \Omega_A) \lbrack X \rbrack)
\end{tikzcd}
\]
commutes. Here, $\iota^{\fp,\fm \cap \Omega_A}: \rk(\fp) \hookrightarrow \rk(\fm \cap \Omega_A)$ is the embedding and $\theta^{\fm \cap \Omega_A}: \Omega_A \rarr \rk(\fm \cap \Omega_A)$ is the canonical morphism. 
\end{proposition}

\begin{proof}
Because of Lemma \ref{gates_containing_omega} and Proposition \ref{dominating_gates} we can assume that $\sO$ contains $\Omega_A$. The claim can now be proven by the same arguments as in the proof of  \cite[Theorem 7.4.3]{GP-Coxeter-Hecke}.
\end{proof}

\begin{definition}
For an $A$-gate $\sO$ in $\fp$ let $\mrm{BN}(\sO)$ be the set of contractions of the maximal ideals of local rings in $K$ dominating $\sO$ and containing $\Omega_A$. This is a non-empty subset of the set of prime ideals of $\Omega_A$ lying over $\fp$.
\end{definition}

\begin{corollary} \label{dec_same_if_over_same_prime}
If $\sO$ and $\sO'$ are two $A$-gates in $\fp$ such that $\mrm{BN}(\sO)$ and $\mrm{BN}(\sO')$ have non-empty intersection, then $\rd_A^{\fp,\sO} = \rd_A^{\fp,\sO'}$.
\end{corollary}

\begin{proof}
By Proposition \ref{dominating_gates} we can assume that $\sO$ and $\sO'$ contain $\Omega_A$ and that their maximal ideals $\fm$ and $\fm'$ contract to the same prime ideal $\fm''$ of $\Omega_A$. We now apply Proposition \ref{brauer_nesbitt_gate} to $\sO$ and $\sO'$. The upper right part of the diagram of this proposition, the composition $\rt_A^{\theta^{\fm''}} \circ \mrm{bn}_A^{(0)}$, is in both cases the same so that due to the commutativity of the diagram we have
\[
\rt_A^{\iota^{\fp ,\fm''}} \circ \mrm{bn}_A^{\fp } \circ \rd_A^{\fp,\sO} = \rt_A^{\iota^{\fp ,\fm''}} \circ \mrm{bn}_A^{\fp } \circ \rd_A^{\fp,\sO'} \;.
\]
The map $\rt_A^{\iota^{\fp ,\fm''}} \circ \mrm{bn}_A^{\fp }$ is injective and thus left cancelable so that we get $\rd_A^{\fp,\sO} = \rd_A^{\fp,\sO'}$.
\end{proof}

An immediate further corollary is the following.

\begin{corollary}
Suppose that $\fp$ is unibranched in the extension $R \subs \Omega_A$, i.e., there is just one prime ideal in $\Omega_A$ lying over $\fp$. Then $\rd_A^{\fp,\sO} = \rd_A^{\fp,\sO'}$ for all $A$-gates $\sO,\sO'$ in $\fp$. The assumption holds if $R = \Omega_A$, and this in turn holds if $R$ is normal.
\end{corollary}

\subsection{The noetherian case} \label{noeth_case}

For our main theorem about the generic behavior of decomposition maps we will need to assume that the base ring is noetherian. The reason for this is that we need \textit{discrete} valuation rings as $A$-gates as only then we can conclude that contractions of $A^K$-modules to $A^\sO$ are $\sO$-free (recall that the difference between a valuation ring and a discrete valuation ring is that for the latter all torsion-free modules are already free and not only the finitely generated ones). The reader will see the importance of discrete valuation rings in most of the following sections and this is why we make the following definition.

\begin{definition}
A \word{discrete} $A$-gate in $\fp$ is an $A$-gate in $\fp$ which is a discrete valuation ring.
\end{definition} 

We are now confronted with two problems: 

\begin{enum_thm}
\item Do there exist discrete $A$-gates?
\item Can we realize any decomposition map by a discrete $A$-gate? 
\end{enum_thm}
We show that the answer to both questions is positive, provided that $R$ is noetherian and $A(\fp)$ splits. This result makes the theory of decomposition maps in the noetherian split case much slicker than before. The first ingredient for proving this is the following theorem by Grothendieck \cite[7.1.7]{Grothendieck.A61Elements-d_EGA2} (see also \cite[15.6]{GorWed10-Algebraic-geomet}) on the existence of dominating discrete valuation rings in the noetherian case.

\begin{theorem}[Grothendieck] \label{discrete_exists}
For any \textit{noetherian} integral domain $R$ with quotient field $K$, any \textit{finitely generated} field extension $L$ of $K$, and any non-zero prime ideal $\fp$ of $R$ there exists a \textit{discrete} valuation ring $\sO$ between $R$ and $L$ with maximal ideal $\fm$ and $R \cap \sO = \fp$.
\end{theorem}

The second ingredient is our result from the last section. It allows us to prove the following theorem.

\begin{theorem} \label{all_decs_are_discrete}
Let $\fp$ be a non-zero prime ideal of $R$. Suppose that $R$ is noetherian and that $A(\fp)$ splits. If $\sO$ is any $A$-gate in $\fp$, then there is a discrete $A$-gate $\sO'$ in $\fp$ with $\rd_A^{\fp,\sO} = \rd_A^{\fp,\sO'}$.
\end{theorem}

\begin{proof}
Let $\fq \in \mrm{BN}(\sO)$ and recall that this is a prime ideal of $\Omega_A$ lying over $\fp$. Furthermore, recall from Lemma \ref{bn_attractor_finite_type} that $\Omega_A$ is an $R$-algebra of finite type. Hence, since $R$ is noetherian, also $\Omega_A$ is noetherian. We can now use Theorem \ref{discrete_exists} to deduce that there exists a discrete valuation ring $(\sO',\fm')$ between $\Omega_A$ and $K$ such that $\Omega_A \cap \fm' = \fq$. This ring is an $A$-gate by Lemma \ref{split_gates_exist}. Since $\fq \in \mrm{BN}(\sO')$ by construction, the claim follows immediately from Corollary \ref{dec_same_if_over_same_prime}. 
\end{proof}

\section{The main theorems} \label{main_theorems}

Our primary aim is to study for a finite free algebra $A$ over an integral domain $R$ topological and geometric properties of the set
\[
\mrm{DecGen}(A) \dopgleich \left\lbrace \fp \in \Spec(R) \ \middle| \begin{array}{l}  \tn{there is an }A\tn{-gate in } \fp \tn{ and } \rd_A^{\fp,\sO} \\  \tn{is trivial for any } A\tn{-gate } \sO \tn{ in } \fp \end{array} \right\rbrace \;.
\]
Here, \textit{trivial} means that $\rd_A^{\fp,\sO}$ induces a bijection between the isomorphism classes of simple modules of $A^K$ and those of $A(\fp)$. Equivalently, the \word{decomposition matrix} $\rD_A^{\fp,\sO}$, i.e., the matrix of $\rd_A^{\fp,\sO}$ in bases of the Grothendieck groups given by systems of simple modules, is a permutation matrix—and thus the identity matrix up to row and column permutations. If this is true and $(S_i)_{i \in I}$ is a system of representatives of the simple $A^K$-modules, then by choosing an arbitrary $A$-gate $\sO$ and $\sO$-free $A^\sO$-forms $\wt{S}_i$ of the $S_i$ we can produce a system of representatives of the simple $A(\fp)$-modules just by reducing the $\wt{S}_i$ modulo the maximal ideal of $\sO$. This is particularly nice if we manage to realize the $S_i$ over $R_\fp$ already in which case we do not even have to find an $A$-gate. In the following we list our main results about $\mrm{DecGen}(A)$. Their proofs will occupy the rest of this paper. \\

First, we have the genericity result.

\begin{theorem} \label{dec_gen_trivial}
Suppose that $R$ is noetherian and that $A^K$ splits. Then $\mrm{DecGen}(A)$ is a neighborhood of the generic point of $\Spec(R)$, i.e., it contains an open subset of $\Spec(R)$ containing the generic point. Hence, decomposition matrices are generically trivial.
\end{theorem}

Next, we have a theorem explaining the role of $A$-gates in the definition of $\mrm{DecGen}(A)$ and showing that they behave as nice as possible with respect to trivial decomposition maps.

\begin{theorem} \label{one_trivial_all_triival}
Suppose that $R$ is noetherian and that both $A^K$ and $A(\fp)$ split. The following are equivalent:
\begin{enum_thm}
\item $\rd_A^{\fp,\sO}$ is trivial for some $A$-gate $\sO$ in $\fp$.
\item $\rd_A^{\fp,\sO}$ is trivial for any $A$-gate $\sO$ in $\fp$, i.e., $\fp \in \mrm{DecGen}(A)$.
\item $\dim_K \Jac(A^K) = \dim_{\rk(\fp)} \Jac(A(\fp))$, where $\Jac$ denotes the Jacobson radical.
\end{enum_thm}
\end{theorem}

Finally, we have the openness result.

\begin{theorem} \label{dec_stratification}
Suppose that $R$ is noetherian and that $A(\fp)$ splits for all $\fp$. Then $\mrm{DecGen}(A)$ is an \textit{open} subset of $X \dopgleich \Spec(R)$ containing the generic point. Moreover, there are finitely many points $\xi_1,\ldots,\xi_n$ in $X$ such that 
\[
X = \bigcup_{i=1}^n \mrm{DecGen}(A|_{\xi_i}) \;.
\]
The sets $\mrm{DecGen}(A|_{\xi_i})$ are locally closed in $X$.  
\end{theorem}

The theorem shows that we can stratify $\Spec(R)$ into finitely many strata on which decomposition maps are trivial and so the fibers of $A$ are ``essentially the same''. 
Moreover, the complement $\mrm{DecEx}(A)$ of $\mrm{DecGen}(A)$ in $\Spec(R)$ is closed and so it is the zero locus of an ideal $\fd_A$ in $\Spec(R)$. We call this ideal the \word{decomposition discriminant} of $A$. The obvious question is: can we explicitly compute $\fd_A$? We cannot give a general answer so far. What we can prove is that in the setting of Schur elements this discriminant is indeed given by the Schur elements of the simple modules—as one might expect.

\begin{proposition} \label{tits_def_theorem_modgen}
Suppose that $R$ is noetherian and normal, that $A$ has split fibers and is symmetric, and that $A^K$ is semisimple. Let $(c_i)_{i \in I}$ be the Schur elements of a system $(S_i)_{i \in I}$ of representatives of the simple $A^K$-modules (see \cite[\S7]{GP-Coxeter-Hecke}). Then
\[
\mrm{DecEx}(A) = \bigcup_{i \in I} \rV(c_i) = \rV(\prod_{i \in I} c_i ).
\] 
In particular, the decomposition discriminant $\fd_A$ is equal to the ideal generated by the Schur elements $c_i$.
\end{proposition}

We note that even though the Schur elements depend on the symmetric structure of $A$ the discriminant $\fd_A$ does not. The proposition above shows that we can view $\fd_A$ as a generalization of Schur elements of simple modules for symmetric semisimple algebras.

\subsection{Examples} \label{examples}

In the following we list some classical and recent examples which fit into our setting.

\begin{example}
We start with the classical setting in modular representation theory of finite groups. Let $G$ be a finite group. Let $\sO$ be the ring of integers in a number field $K$ which is sufficiently large for $G$, and let $A = \sO G$ be the group algebra. The specialization of $A$ in the generic point $(0)$ of $\Spec(\sO)$ is the characteristic zero group algebra $KG$ and the specialization of $A$ in a prime ideal $\fp \neq (0)$ of $\sO$ is the positive characteristic group algebra $\rk(\fp)G$ of $G$ over a finite field of characteristic $p$, where $(p) = \bbZ \cap \fp$. We note that it is a classical fact that $A$ has split fibers. The decomposition maps $\rd_A^\fp$ for $\fp \in \Spec(\sO)$ are precisely those considered in modular representation theory of finite groups and it is a classical fact due to Brauer that $\rd_A^\fp$ is non-trivial if and only if $\fp$ lies above a prime number $p$ dividing the order of $G$, i.e., $\mrm{DecEx}(A) = \bigcup_{p \mid \# G} \rV(p)$. This is exactly the statement of Proposition \ref{tits_def_theorem_modgen} which can be applied here.
\end{example}

\begin{example}
 Let $n \in \bbN$. For any $\delta \in \bbC$ the Brauer algebra $B_n(\delta)$ is a $\bbC$-algebra with basis certain diagrams whose multiplication involves the parameter $\delta$ (see \cite{Brauer:1937aa}). We can combine all these algebras by taking an indeterminate  $\boldsymbol{\delta}$ over $\mathbb{C}$ and considering the \textit{generic Brauer algebra} $\mathbf{B}_n^0 \dopgleich B_n(\boldsymbol{\delta})$, i.e., the Brauer algebra over $\mathbb{C} \lbrack \boldsymbol{\delta} \rbrack$ with parameter $\boldsymbol{\delta}$. The specialization of $\mathbf{B}_n^0$ in the maximal ideal $(\boldsymbol{\delta}-\delta)$ of $\mathbb{C} \lbrack \boldsymbol{\delta} \rbrack$ corresponding to $\delta \in \mathbb{C}$ is precisely the Brauer algebra $B_n(\delta)$. Note that the specialization of $\mathbf{B}_n^0$ in the generic point $(0)$ is the Brauer algebra over the rational function field $\bbC (\boldsymbol{\delta})$. Brauer algebras are examples of cellular algebras in the sense of Graham and Lehrer \cite{Graham:1996aa}, and these split over any field. Hence, $\mathbf{B}_n^0$ has split fibers. Our result implies that the decomposition maps $\mathbf{B}_n^0$ are generically trivial, hence trivial for all but finitely many $\fp \in \Spec(\mathbb{C} \lbrack \boldsymbol{\delta} \rbrack)$. Theorem \ref{one_trivial_all_triival} implies in particular that $B_n(\delta)$ is semisimple for all but finitely many $\delta \in \bbC$. Due to its generality our genericity result does not explicitly give us the $\delta$ for which $B_n(\delta)$ is not semisimple. These have been explicitly determined by Wenzl \cite{Wenzl:1988aa} and Rui \cite{Rui:2005aa}.
\end{example}

\begin{example}
One cannot expect a lot of geometry from the one-dimensional base ring $\bbC \lbrack \boldsymbol{\delta} \rbrack$ in the above example. We can, however, also study modular Brauer algebras by considering the generic Brauer algebra $\mathbf{B}_n \dopgleich B_n(\boldsymbol{\delta})$ over $\bbZ \lbrack \boldsymbol{\delta} \rbrack$. This has now a two-dimensional base ring and here we can expect some geometry (see \cite{Mumford:1999aa} for a nice illustration of the spectrum of $\mathbb{Z}\lbrack \boldsymbol{\delta} \rbrack$). Again note that $\mathbf{B}_n$ has split fibers. Specializing in the prime ideal $(\boldsymbol{\delta} - \delta)$ for $\delta \in \bbZ$ yields the characteristic zero Brauer algebra over $\bbQ$ in $\delta$ as above. Specializing in the prime ideal $(p)$ for a prime number $p$ yields the (generic) modular Brauer algebra over the rational function field $\bbF_p(\boldsymbol{\delta})$. Specializing in the maximal ideal $(p,\delta)$ gives the Brauer algebra for the image of $\delta$ in $\bbF_p$. We see here already that the ability to specialize in non-closed points of the spectrum of the base ring  is useful. Our result implies that decomposition maps of $\mathbf{B}_n$ are trivial on an open subset of $\bbA_\bbZ^2$. An explicit description of this subset follows from the semisimplicity criterion by Rui \cite{Rui:2005aa}.
\end{example}

\begin{example}
Let $W$ be an irreducible finite reflection group acting on a finite-dimensional vector space over a field $k$ of characteristic zero. Let $c:\mrm{Ref}(W) \rarr k$ be a function from the set $\mrm{Ref}(W)$ of reflections in $W$ to $k$ which is invariant under the action of $W$ so that $c$ descends to a function $\mrm{Ref}(W)/W \rarr k$. To any such function the restricted rational Cherednik algebra $\ol{\rH}_c$ is defined (see \cite{Gordon:2003aa}). This is a finite-dimensional $k$-algebra obtained as a quotient of the correspond rational Cherednik algebra at $t=0$ introduced by Etingof and Ginzburg \cite{EG-Symplectic-reflection-algebras}. One can obtain them as specializations of the generic restricted rational Cherednik algebra $\ol{\bH}$ in closed points of $\bbA_k^{n}$, where $n \dopgleich \# (\mrm{Ref}(W)/W)$. This algebra is free and of finite dimension over $k \lbrack \mathbf{c} \rbrack$, where $\mathbf{c}$ is a family of $n$ algebraically independent elements over $k$ (see \cite{Bonnafe.C;Rouquier.R13Cellules-de-Calogero} and \cite{Thiel:2015aa}). Moreover, it has split fibers (see loc. cit.). Our result implies that the decomposition maps of $\ol{\bH}$ are generically trivial, so that the number and dimensions of the simple modules of $\ol{\rH}_c$ are constant for all $c$ in an open subset of $\bbA_k^n$. As explained in the introduction, this case is in general not covered by Proposition \ref{tits_def_theorem_modgen} as the generic fiber of $\ol{\bH}$ is semisimple if and only if $W$ is a cyclic group. We do not have an explicit description of the decomposition discriminant $\fd_{\ol{\bH}}$ for general $W$.
\end{example}

\begin{example}  \label{Cherednik_t0_example}
The actual infinite-dimensional rational Cherednik algebra $\rH_c$ at $t=0$ attached to a reflection group $W$ as above is finite and free over the central subalgebra $\mathcal{Z} \dopgleich k \lbrack V \rbrack^W \otimes_k k \lbrack V^* \rbrack^W$. Again this algebra can be viewed as a specialization of the generic rational Cherednik algebra $\bH$ at $t=0$ which is a finite free $\boldsymbol{\mathcal{Z}}$-algebra with $\boldsymbol{\mathcal{Z}} \dopgleich k \lbrack \mathbf{c} \rbrack \otimes_k \mathcal{Z}$ (see \cite{Bonnafe.C;Rouquier.R13Cellules-de-Calogero}). Note that the specialization of $\bH$ in the origin of $\boldsymbol{\mathcal{Z}}$ yields the generic fiber of the generic restricted rational Cherednik algebra $\ol{\bH}$ considered above. The generic fiber of $\bH$ is not split (see \cite{Bonnafe.C;Rouquier.R13Cellules-de-Calogero}) and so we cannot apply our main result to $\bH$. In their construction of Calogero–Moser cells Bonnafé and Rouquier \cite{Bonnafe.C;Rouquier.R13Cellules-de-Calogero} consider an extension of $\bH$. Let $\bM$ be the Galois closure of the (separable) field extension $\mrm{Frac}(\boldsymbol{\mathcal{Z}}) \subs \mrm{Frac}(\rZ(\bH))$ and let $\bR$ be the integral closure of $\boldsymbol{\mathcal{Z}}$ in $\bM$. Then $\wt{\bH} \dopgleich \bH^\bR$ is a finite free $\bR$-algebra with split generic fiber. As $\boldsymbol{\mathcal{Z}}$ is a polynomial ring and thus a Japanese ring, the ring $\bR$ is noetherian and so our result implies that the decomposition maps of $\wt{\bH}$ are generically trivial. Decomposition maps of $\wt{\bH}$ play an important role in \cite{Bonnafe.C;Rouquier.R13Cellules-de-Calogero} in the context of Calogero–Moser families and Calogero–Moser cells (a conjectural extension of Kazhdan–Lusztig cells to complex reflection groups).
\end{example}

\begin{example}
Let $G$ be a connected reductive algebraic group over an algebraically closed field $k$ of characteristic $p$ and let $\fg$ be the Lie algebra of $G$. As in \cite{Brown-Gordon-Ramifications-2001} we assume that the derived group of $G$ is simply-connected, that $p$ is an odd good prime for $G$, and that $\fg$ has a non-degenerate $G$-invariant bilinear form. Let $U \dopgleich U(\fg)$ be the enveloping algebra of $\fg$ and let $\mathcal{Z} \dopgleich k \lbrack x^p - x^{\lbrack p \rbrack}  \mid x \in \fg \rbrack \subs \rZ(U)$ be its $p$-center. Then $U$ is a finite free $\mathcal{Z}$-module by the Poincaré–Birkhoff–Witt theorem. As in Example \ref{Cherednik_t0_example} the generic fiber of $U$ over $\mathcal{Z}$ is not split and one can split it over an extension $\widetilde{\mathcal{Z}}$ of $\mathcal{Z}$. Our result is then applicable to the extension $\widetilde{U}$ of $U$ to $\widetilde{\mathcal{Z}}$. Note that the fibers of $\widetilde{U}$ in closed points of $\widetilde{\mathcal{Z}}$ are isomorphic to the fibers of $U$ in closed points of $\mathcal{Z}$ so that even though we are considering an extension of $U$ we still get information about $U$ itself.
\end{example}

Similar as Brauer algebras and rational Cherednik algebras one can also consider Hecke algebras (for complex reflection groups in general) over appropriate base rings which admit interesting specializations. We refer to \cite{GP-Coxeter-Hecke},  \cite{Geck.M;Jacon.N11Representations-of-H}, and \cite{BMR-Complex-Reflection}. As mentioned in the introduction, Geck \cite{Geck.M98Representations-of-H} used generic triviality of decomposition maps in the context of James's conjecture.

\section{Connection with the Jacobson radical} \label{connection_with_jac}

Working directly with decomposition maps is a difficult problem. Our strategy—based on arguments by Geck \cite{Geck.M98Representations-of-H} in a one-dimensional setting—is to use a connection to the behavior of the Jacobson radical under specialization. The starting point is the following proposition. We denote by $\Jac(A)$ the Jacobson radical of a ring $A$. Furthermore, $\fp$ denotes a non-zero prime of $R$.

\subsection{Implication of a decomposition map being trivial}

\begin{proposition} \label{d_gen_jac_gen}
Suppose that $A^K$ and $A(\fp)$ split. If $\rd_A^{\fp,\sO}$ is trivial for an $A$-gate $\sO$ in $\fp$, then $\dim_K \Jac(A^K) = \dim_{\rk(\fp)} \Jac(A(\fp))$.
\end{proposition}

\begin{proof}
Let $(S_i)_{i \in I}$ be a system of representatives of the isomorphism classes of simple $A^K$-modules. Then by assumption $(\rd_A^{\fp,\sO}(\lbrack S_i \rbrack))_{i \in I}$ is a system of representatives of the isomorphism classes of simple $A(\fp)$-modules. Note that $\rd_A^{\fp,\sO}$ preserves dimensions by construction. Since both $A^K$ and $A(\fp)$ split, we have by \cite[7.8]{Lam-Noncommutative} the equality
\begin{align*}
\dim_K \Jac(A^K) + \sum_{i \in I} (\dim_K S_i)^2 & = \dim_K A^K = \dim_{\rk(\fp)} A(\fp) \\ & = \dim_{\rk(\fp)} \Jac(A(\fp)) + \sum_{i \in I} ( \dim_{\rk(\fp)} \rd_A^{\fp,\sO}(\lbrack S_i \rbrack)) ^2 \\ & =  \dim_{\rk(\fp)} \Jac(A(\fp)) + \sum_{i \in I} ( \dim_{K} S_i ) ^2   
\end{align*}
and therefore 
$
\dim_K \Jac(A^K) = \dim_{\rk(\fp)} \Jac(A(\fp))
$
as claimed.
\end{proof}

Note that the proposition shows that triviality of the decomposition map for some $A$-gate implies a statement not involving the $A$-gate anymore. Also note that it implies that if $A^K$ splits, then 
\begin{equation} \label{decgen_spl_subs_jacdimgen_spl}
\mrm{DecGen}(A) \cap \mrm{SplGen}(A) \subs \ul{\mrm{JacDim}}\mrm{Gen}(A) \cap \mrm{SplGen}(A)  \;,
\end{equation}
where
\[
\ul{\mrm{JacDim}}\mrm{Gen}(A)  \dopgleich \lbrace \fp \in \Spec(R) \mid \dim_K \Jac(A^K) = \dim_{\rk(\fp)} \Jac(A(\fp)) \rbrace 
\]
and
\[
\mrm{SplGen}(A) \dopgleich \lbrace \fp \in \Spec(R) \mid A(\fp) \tn{ splits} \rbrace \;.
\]
We will see in \S\ref{generic_representation_theory} why we used an underline in the notation. If the converse inclusion of (\ref{decgen_spl_subs_jacdimgen_spl}) would hold and we could show that $\ul{\mrm{JacDim}}\mrm{Gen}(A) \cap \mrm{SplGen}(A)$ is a neighborhood of the generic point, then also $\mrm{DecGen}(A)$ would be a neighborhood of the generic point so that decomposition maps would be generically trivial. Our aim in this paragraph is to show that the converse inclusion indeed holds—provided that $R$ is noetherian. In the following paragraphs we then show the desired geometric properties. 

\subsection{Scalar extension of submodules} \label{submodule_stuff}

To understand the set $\ul{\mrm{JacDim}}\mrm{Gen}(A)$ we need a way to compare the Jacobson radical of the generic fiber with the Jacobson radical of a specialization. To this end, we need some results about the behavior of submodules under scalar extension. If $\theta:R \rarr S$ is a morphism into a commutative ring $S$, we set $A^S \dopgleich \theta^*A = S \otimes_R A$. The natural ring morphism $\theta_A:A \rarr A^S$, $a \mapsto 1 \otimes a$, induces the scalar extension functor $\theta_A^*:\cat{Mod}{A} \rarr \cat{Mod}{A^S}$. If $V$ is an $A$-module, then $\theta_A^* V = A^S \otimes_A V$ is an $A^S$-module whose underlying $S$-module structure is the one of $V^S$. Furthermore, we have a natural map $\theta_V:V \rarr V^S$, $v \mapsto 1 \otimes v$, and this map allows us to set up a relation between $A$-submodules of $V$ and $A^S$-submodules of $V^S$. Namely, if $U \leq V$, we set $\mrm{ext}_V^S(U) \dopgleich \langle \theta_V(U) \rangle_{A^S} \leq V^S$ and $\mrm{con}_V^S(W) \dopgleich \theta_V^{-1}(W) \leq V$ for $W \leq V^S$. 
If $\theta_V$ is injective, we can identify $V$ with a subset of $V^S$ and then we have $\mrm{con}_V^S(W) = V \cap W$. This holds for example if $\theta$ is the localization morphism for a multiplicatively closed subset $\Sigma \subs R$ and $V$ is $\Sigma$-torsion free. In this case \cite[II, \S2.2, Proposition 4]{Bou-Commutative-Algebra-1-7} implies moreover that $\mrm{ext}_V^S \circ \mrm{con}_V^S(W) = W$ for any $W \leq V^S$ and that the image of $\mrm{con}_V^S$ consists of all submodules $U \leq V$ such that the quotient $V/U$ is $\Sigma$-torsion free. 

Using right-exactness of $\theta_A^*$ it is easy to see that if $f:V \rarr W$ is a surjective morphism of $A$-modules, then $\Im(\theta_A^*f) = \theta_A^*W$ and $\Ker(\theta_A^*f) = \mrm{ext}_V^S(\Ker(f))$. In particular, if $U$ is a submodule of $V$, then $\theta_A^*(V/U) \cong \theta_A^*V/\mrm{ext}_V^S(U)$ canonically. We will use this several times later.

\subsection{A generalization of Tits's deformation theorem} \label{tits_section}

If $J$ is a $K$-subspace of $A^K$, then (due to the canonical $A$-form $A$ of $A^K$) we can construct for any $\fp \in \Spec(R)$ canonically a subvector space $J(\fp)$ of $A(\fp)$ derived from $J$, namely
\[
J(\fp) \dopgleich \mrm{ext}_{A_\fp}^{\rk(\fp)} \circ \mrm{con}_{A_\fp}^K(J) = \rk(\fp) \otimes_{R_\fp} ( A_\fp \cap J ) \;.
\] 
The following lemma sets up a first relation between $\Jac(A^K)$ and $\Jac(A(\fp))$.

\begin{lemma} \label{jac_spec_dim_lemma}
For any $\fp \in \Spec(R)$ the relation 
$
 \Jac(A(\fp)) \sups \Jac(A^K)(\fp) 
$
 holds.
\end{lemma}

\begin{proof}
Since $A^K$ is artinian, its Jacobson radical is nilpotent. Hence, $\mrm{con}_{A_\fp}^K(\Jac(A^K)) = A_\fp \cap \Jac(A^K)$ is nilpotent, and so $\Jac(A^K)(\fp)$ is nilpotent. This implies that $\Jac(A^K)(\fp)$ is contained in $\Jac(A(\fp))$.
\end{proof}

The following theorem is a generalization of Tits's deformation theorem as in \cite[7.4.6]{GP-Coxeter-Hecke} to a non-semisimple situation. For this argument to work we really need \textit{discrete} valuation rings.

\begin{theorem} \label{tits_deformation_geck}
Suppose that $A^K$ and $A(\fp)$ split and let $\sO$ be a \textit{discrete} $A$-gate in $\fp$. If $\dim_{\rk(\fp)} \Jac(A(\fp)) = \dim_K \Jac(A^K)$, then $\rd_A^{\fp,\sO}$ is trivial.
\end{theorem}

\begin{proof} 
Let $\fm$ be the maximal ideal of $\sO$. We first show that the Jacobson radical of the specialization $A^\sO(\fm)$ is equal to the specialization of $\Jac(A^K)$ in $\fm$. By Lemma \ref{jac_spec_dim_lemma} we have 
\begin{equation} \label{jac_AO_m_subs_Jac_AOm}
\mrm{ext}_{A^\sO}^{\rk(\fm)} \circ \mrm{con}_{A^\sO}^K \Jac(A^K) \subs \Jac(A^\sO(\fm)) \;.
\end{equation}
Since $A^\sO$ is $\sO$-torsion free, also $\mrm{con}_{A^\sO}^K(\Jac(A^K)) \leq A^\sO$ is $\sO$-torsion free and thus free since $\sO$ is a discrete valuation ring. Moreover, by \S\ref{submodule_stuff} we have
\[
\mrm{ext}_{A^\sO}^K \circ \mrm{con}_{A^\sO}^K(\Jac(A^K)) =  \Jac(A^K)
\]
and therefore 
\begin{align} \label{jac_AO_m_dim_jac_AK_dim}
\dim_{\rk(\fm)} \mrm{ext}_{A^\sO}^{\rk(\fm)} \circ \mrm{con}_{A^\sO}^K \Jac(A^K) = \dim_\sO \mrm{con}_{A^\sO}^K \Jac(A^K)  = \dim_K \Jac(A^K) \;.
\end{align}
Since $A(\fp)$ splits, it follows from \cite[7.9(i)]{CR-Methods-1} that 
\[
\mrm{ext}_{A(\fp)}^{\rk(\fm)} \Jac(A(\fp)) = \Jac(A^\sO(\fm)) \;.
\]
 Combining this with equation (\ref{jac_AO_m_dim_jac_AK_dim}) and our assumption, we see that
 \[
 \dim_{\rk(\fm)} \Jac(A^\sO(\fm)) = \dim_{\rk(\fp)} \Jac(A(\fp)) = \dim_K \Jac(A^K) = \dim_{\rk(\fm)} \Jac(A^K)(\fm)  \;.
 \]
Because of this, we must already have equality in (\ref{jac_AO_m_subs_Jac_AOm}), so
\begin{equation} \label{jac_AO_m_eq_jac_AK_m}
\Jac(A^K)(\fm) = \mrm{ext}_{A^\sO}^{\rk(\fm)} \circ \mrm{con}_{A^\sO}^K \Jac(A^K) = \Jac(A^\sO(\fm)) \;.
\end{equation}

We know from \S\ref{submodule_stuff} that the quotient $\wt{A} \dopgleich A^\sO/\mrm{con}_{A^\sO}^K \Jac(A^K)$ is $\sO$-torsion free and thus already $\sO$-free of finite dimension since $\sO$ is a discrete valuation ring. According to \S\ref{submodule_stuff} we have 
\[ 
\wt{A}^K = (A^\sO / \mrm{con}_{A^\sO}^K \Jac(A^K))^K \cong A^K/(\mrm{ext}_{A^\sO}^K \circ \mrm{con}_{A^\sO}^K \Jac(A^K)) = A^K/\Jac(A^K)  \;.
\]
Hence, $\wt{A}^K$ is semisimple, and it is also split as a quotient the split algebra $A^K$. Furthermore, the canonical morphism
\[
\rG_0(A^K) \rarr \rG_0(A^K/\Jac(A^K)) = \rG_0(\wt{A}^K)
\]
clearly is trivial, i.e., inducing a bijection between the simple modules. Because of equation (\ref{jac_AO_m_eq_jac_AK_m}) also 
\begin{align*}
\wt{A}(\fm) &= \wt{A}^{\rk(\fm)} = (A^\sO/\mrm{con}_{A^\sO}^K \Jac(A^K) )^{\rk(\fm)} \cong A^\sO(\fm)/(\mrm{ext}_{A^\sO}^{\rk(\fm)} \circ \mrm{con}_{A^\sO}^K \Jac(A^K)) \\ & = A^\sO(\fm)/\Jac(A^\sO(\fm))
\end{align*}
is split semisimple and again the canonical morphism
\[
\rG_0(A^\sO(\fm)) \rarr \rG_0(A^\sO(\fm)/\Jac(A^\sO(\fm))) = \rG_0(\wt{A}(\fm))
\]
is trivial. Moreover, since $A(\fp)$ splits, the canonical morphism 
\[
\rG_0(A(\fp)) \rarr \rG_0(A(\fp)^{\rk(\fm)}) = \rG_0(A^\sO(\fm))
\]
is trivial and so we have a canonical morphism
\[
\rG_0(A(\fp)) \rarr \rG_0(\wt{A}(\fm))
\]
which is trivial.

Now, $\sO$ is an $\wt{A}$-gate in $\fp$ and so the decomposition map 
\[
\rd_{\wt{A}}^{\fm,\sO}:\rG_0(\wt{A}^K) \rarr \rG_0(\wt{A}(\fm))
\]
exists. According to Tits's deformation theorem \cite[7.4.6]{GP-Coxeter-Hecke} (for which we do not need that the base ring is normal, we just need the decomposition morphism to exist as follows directly from the given proof), this morphism is trivial since $\wt{A}(\fm)$ is split semisimple. Once we know that the diagram
\[
\begin{tikzcd}
\rG_0(A^K) \arrow{r}{\rd_A^{\fp,\sO}} \arrow{d}[swap]{\cong} & \rG_0(A(\fp)) \arrow{d}{\cong} \\
\rG_0(\wt{A}^K) \arrow{r}[swap]{\rd_{\wt{A}}^{\fm,\sO}} & \rG_0(\wt{A}(\fm))
\end{tikzcd}
\]
commutes, where the vertical morphisms are the morphisms discussed above, we also know that $\rd_A^{\fp,\sO}$ is trivial. It suffices to check commutativity on simple $A^K$-modules, so let $S$ be a simple $A^K$-module. Let $\wt{S}$ be an $\sO$-free $A^\sO$-form of $S$. Then $\rd_A^{\fp,\sO}(\lbrack S \rbrack) = (\gamma_{A}^{\fp,\fm})^{-1}( \lbrack \wt{S}/\fm \wt{S} \rbrack)$, and the image of this element under the right vertical morphism is equal to $\lbrack \wt{S}/\fm \wt{S} \rbrack$. On the other hand, the image of $\lbrack S \rbrack$ under the left vertical morphism is again $\lbrack S \rbrack$ and as $\wt{S}$ is also an $\sO$-free $\wt{A}$-form of $S$, we have $\rd_{\wt{A}}^{\fp,\sO}(\lbrack S \rbrack) = \lbrack \wt{S}/\fm \wt{S} \rbrack$. Hence, the diagram commutes and therefore $\rd_A^{\fp,\sO}$ is trivial as claimed.
\end{proof}

\subsection{Consequences}

Using the two results about the connection between triviality of a decomposition map and preservation of the dimension of the Jacobson radical we can now prove that triviality of a decomposition map in $\fp$ does not depend on the choice of the gate, provided that $R$ is noetherian and both $A^K$ and $A(\fp)$ split.

\begin{proof}[{Proof} of Theorem \ref{one_trivial_all_triival}]
Suppose that $\rd_A^{\fp,\sO}$ is trivial for some $A$-gate $\sO$ in $\fp$. Proposition \ref{d_gen_jac_gen} implies that $\dim_K \Jac(A^K) = \dim_{\rk(\fp)} \Jac(A(\fp))$. Now, let $\sO'$ be an arbitrary $A$-gate in $\fp$. By Theorem \ref{all_decs_are_discrete} there is a discrete $A$-gate $\sO''$ in $\fp$ with $\rd_A^{\fp,\sO'} = \rd_A^{\fp,\sO''}$. Since $\dim_K \Jac(A^K) = \dim_{\rk(\fp)} \Jac(A(\fp))$ and $\sO''$ is discrete, Theorem \ref{tits_deformation_geck} implies that $\rd_A^{\fp,\sO''}$ is trivial. Hence, $\rd_A^{\fp,\sO'}$ is trivial. 
\end{proof}

Putting all results of this section together we can now conclude that the inclusion (\ref{decgen_spl_subs_jacdimgen_spl}) is actually an equality.

\begin{corollary} \label{decgen_cap_spl_eq_jacdimgen_cap_spl}
If $A^K$ splits and if $R$ is noetherian, then
\[
\mrm{DecGen}(A) \cap \mrm{SplGen}(A) = \ul{\mrm{JacDim}}\mrm{Gen}(A) \cap \mrm{SplGen}(A) \;.
\] 
\end{corollary}

Note that if $A$ has split fibers, the corollary shows that $\mrm{DecGen}(A) = \ul{\mrm{JacDim}}\mrm{Gen}(A) $. In a setting where we can use Schur elements (this is the setting of Proposition \ref{tits_def_theorem_modgen}), we can now already conclude that $\mrm{DecGen}(A)$ is open.

\begin{proof}[{Proof} of Proposition \ref{tits_def_theorem_modgen}]
As $R$ is normal and $A$ is symmetric, the Schur elements are contained in $R$ by \cite[7.3.9]{GP-Coxeter-Hecke} and so $\rV(c_i)$ is well-defined. Suppose that $\fp$ is not contained in any of the $\rV(c_i)$. Then the images of the $c_i$ in $A(\fp)$ are all non-zero and so it follows from Tits's deformation theorem \cite[7.4.6]{GP-Coxeter-Hecke} that $A(\fp)$ is semisimple and that $\rd_A^\fp$ is trivial. Hence, $\fp \in \mrm{DecGen}(A)$. This shows that
\[
\mrm{DecGen}(A) \sups \Spec(R) \setminus \bigcup_{i \in I} \rV(c_i) \;.
\]
Now, let $\fp \in \bigcup_{i \in I} \rV(c_i)$. Then one Schur element of $A(\fp)$ is equal to zero and therefore $A(\fp)$ is not semisimple by \cite[7.2.6]{GP-Coxeter-Hecke}. Hence, $\Jac(A(\fp)) \neq 0$ and therefore $\fp \notin \ul{\mrm{JacDim}}\mrm{Gen}(A) = \mrm{DecGen}(A)$ by Corollary \ref{decgen_cap_spl_eq_jacdimgen_cap_spl}. This shows the asserted equality.
\end{proof}

As explained in the introduction, the setting of  Proposition \ref{tits_def_theorem_modgen} is too restrictive for us. We will see that we can remove all the assumptions that $R$ is normal, that $A$ is symmetric, and that $A^K$ is semisimple, and can still conclude that $\mrm{DecGen}(A)$ is open—although we loose the explicit description of this set. To this end, we have to establish geometric properties of the sets on the right hand side of Corollary \ref{decgen_cap_spl_eq_jacdimgen_cap_spl} in general. In the next two sections we first show that these sets are neighborhoods of the generic point. In \S\ref{generic_representation_theory} we turn to the problem of proving openness. Note already that all these sets are formed by considering a property on a class of algebras and then sort out for given $A$ all fibers satisfying this property. This will be the point of view in \S\ref{generic_representation_theory}.

\section{The split locus} \label{split_locus}

In this paragraph we show that the \textit{split locus} $\mrm{SplGen}(A)$ is a generic neighborhood for any finite free algebra $A$ over an integral domain. We note that this was already proven in \cite{Cline:1999aa} with less effort but our proof contains already some key arguments to be used in the proof of the main theorem. \\

We will use the following result due to Bonnafé and Rouquier which is proven in \cite[proposition C.2.11]{Bonnafe.C;Rouquier.R13Cellules-de-Calogero} in the context of the behavior of blocks under specializations. We give it in a more general form here but prove it by the same arguments.

\begin{lemma} \label{br_gen_lemma} \label{symb_genAF}
Let $\sF \subs A^K$ be a finite subset. Then 
\[
 \mrm{Gen}_A(\sF) \dopgleich \lbrace \fp  \in \Spec(R) \mid \sF \subs A_{\fp } \rbrace 
\]
is an open generic neighborhood in $\Spec(R)$.
\end{lemma}

\begin{proof}
For an element $\alpha \in K$ we define $I_\alpha \dopgleich \lbrace r \in R \mid r\alpha \in R \rbrace$. This is an ideal in $R$ and it has the property that $\alpha \in R_{\fp }$ if and only if $I_\alpha \nsubseteq \fp $. To see this, suppose that $\alpha \in R_{\fp }$. Then we can write $\alpha = \frac{r}{x}$ for some $x \in R \setminus \fp $. Hence, $x \alpha = r \in R$ and therefore $x \in I_\alpha$. Since $x \notin \fp $, it follows that $I_\alpha \nsubseteq \fp $. Conversely, if $I_\alpha \nsubseteq \fp $, then there exists $x \in I_\alpha$ with $x \notin \fp $. By definition of $I_\alpha$ we have $x \alpha \gleichdop r \in R$ and since $x \notin \fp $, we can write $\alpha = \frac{r}{x} \in R_{\fp }$. 

Now, let $(a_1,\ldots,a_n)$ be an $R$-basis of $A$. Then we can write every element $f \in \sF$ as $f = \sum_{i=1}^n \alpha_{f,i} a_i$ with $\alpha_{f,i} \in K$. Let
\[
I \dopgleich \prod_{{f \in \sF, \; i \in \lbrack 1,n \rbrack}} I_{\alpha_{f,i}} \unlhd R \;.
\] 
Then by the properties of the ideals $I_\alpha$ we have the following logical equivalences:
\[
\begin{array}{rcl}
(\sF \subs A_{\fp }) & \Longleftrightarrow & (\alpha_{f,i} \in R_{\fp } \quad \forall f \in \sF, i \in \lbrack 1,n \rbrack) \\ &\Longleftrightarrow& (I_{\alpha} \not\subs \fp  \quad \forall f  \in \sF, i \in \lbrack 1,n \rbrack) \\ & \Longleftrightarrow & (I \not\subs \fp ),
\end{array}
\]
the last equivalence following from the fact that $\fp $ is prime. Hence, 
\[
\Spec(R) \setminus \mrm{Gen}_A(\sF) = \rV(I) \;,
\]
implying that $\mrm{Gen}_A(\sF)$ is an open subset of $\Spec(R)$.
\end{proof}

\begin{corollary} \label{br_gen_lemma_infinite}
If $\sF \subs A^K$ is a (not necessarily finite) subset, then the set
\[
\lbrace \fp \in \Spec(R) \mid \sF \cap A_\fp \neq \emptyset \rbrace
\]
is an open generic neighborhood of $\Spec(R)$.
\end{corollary}

\begin{proof}
The given set is clearly equal to $\bigcup_{f \in \sF} \mrm{Gen}_A(\lbrace f \rbrace)$ 
and as each set $\mrm{Gen}_A(\lbrace f \rbrace)$ is an open generic neighborhood by Lemma \ref{br_gen_lemma}, so is the given set.
\end{proof}

\begin{remark} \label{codim1}
The proof by Bonnafé and Rouquier actually also shows that the complement $\mrm{Ex}_A(\sF)$ of $\mrm{Gen}_A(\sF)$ is either empty or pure of codimension one.
\end{remark}

Our proof of the fact that $\mrm{SplGen}(A)$ is a generic neighborhood is based on arguments by Geck \cite{Geck.M98Representations-of-H}. The key idea is to consider the behavior of  morphisms from $A^K$ into split semisimple $K$-algebras upon reduction modulo prime ideals of $R$. The following proposition shows that the set of prime ideals where such a morphism has ``good reduction'' is indeed open.

\begin{proposition} \label{surjective_morphism_open_gen} \label{symb_genApsi}
Let $\psi:A^K \rarr \prod_{t=1}^n \Mat_{n_t}(K)$ be a surjective morphism of $K$-modules. When considering $\prod_{t=1}^n \Mat_{n_t}(R_\fp)$ canonically as a subset of $\prod_{t=1}^n \Mat_n(K)$, then the sets
\[
\mrm{Gen}_A^{\subs}(\psi) \dopgleich \lbrace \fp \in \Spec(R) \mid \psi(A_\fp) \subs \prod_{t=1}^n \Mat_{n_t}(R_\fp)  \rbrace \;,
\]
\[
\mrm{Gen}_A^{\sups}(\psi) \dopgleich  \lbrace \fp \in \Spec(R) \mid \psi(A_\fp) \sups \prod_{t=1}^n \Mat_{n_t}(R_\fp) \rbrace \;,
\]
and
\[
\mrm{Gen}_A(\psi) \dopgleich \lbrace \fp \in \Spec(R) \mid \psi(A_\fp) = \prod_{t=1}^n \Mat_{n_t}(R_\fp) \rbrace
\]
are open generic neighborhoods in $\Spec(R)$.
\end{proposition}

\begin{proof}
Let $\sB \dopgleich (b_i)_{i=1}^m$ be a basis of $A^K$ such that $(b_i)_{i=1}^r$ is a basis of $\Ker(\psi)$ and $(\psi(b_i))_{i=r+1}^m$ is an $R$-basis of $\prod_{t=1}^n \Mat_{n_t}(R) \subs \prod_{t=1}^n \Mat_{n_t}(K)$. This is possible since $A^K/\Ker(\psi) \cong \prod_{t=1}^n \Mat_{n_t}(K)$ and so one can choose $(b_i)_{i=r+1}^m$ to map for example in each component to the elementary matrices. 

To prove the assertion for the first set, let $\sA \dopgleich (a_i)_{i=1}^m$ be an $R$-basis of $A$. The $K$-linearity of $\psi$ and the fact that $(a_i)_{i=1}^m$ is also an $R_\fp$-basis of $A_\fp$ for all $\fp \in \Spec(R)$ implies the equality
\[
\mrm{Gen}_A^\subs(\psi) = \lbrace \fp \in \Spec(R) \mid \psi(a_i) \in \prod_{t=1}^n \Mat_{n_t}(R_\fp) \quad \tn{for all }  i =1,\ldots,m \rbrace \;.
\]
We can write each basis element $a_i$ uniquely as $a_i = \sum_{j=1}^m \alpha_{ij} b_j$ with $\alpha_{ij} \in K$. As $(\psi(b_i))_{i=r+1}^m$ is an $R$-basis of $\prod_{t=1}^n \Mat_{n_t}(R)$, it is also an $R_\fp$-basis of $\prod_{t=1}^n \Mat_{n_t}(R_\fp)$ and since
\[
\psi(a_i) = \psi \left( \sum_{j=1}^m \alpha_{ij} b_j \right) = \sum_{j=1}^m \alpha_{ij} \psi(b_i) = \sum_{j=r+1}^m \alpha_{ij} \psi(b_i) \;,
\]
the uniqueness of the basis representation implies that the element $\psi(a_i)$ is contained in $\prod_{t=1}^n \Mat_{n_t}(R_\fp)$ if and only if $\alpha_{ij} \in R_\fp$ for all $i = 1,\ldots,m$ and all $j=r+1,\ldots,m$. Hence,
\[
\mrm{Gen}_A^\subs(\psi) = \mrm{Gen}_R( (\alpha_{ij})_{\substack{i=1,\ldots,m \\ j=r+1,\ldots,m}})
\]
and this is an open generic neighborhood by Lemma \ref{br_gen_lemma}.

Now we consider the assertion for the second set. It is obvious that
\begin{align*}
\mrm{Gen}_A^\sups(\psi) & = \bigcap_{x \in \prod_{t=1}^n \Mat_{n_t}(R_\fp)} \lbrace \fp \in \Spec(R) \mid \psi^{-1}(x) \cap A_\fp \neq \emptyset \rbrace \\
& \subs \bigcap_{i=r+1}^m \lbrace \fp \in \Spec(R) \mid \psi^{-1}(\psi(b_i)) \cap A_\fp \neq \emptyset \rbrace \;.
\end{align*}
But this is actually an equality. To see this, suppose that $\fp$ is contained in the last (finite) intersection. Then we can choose $c_i \in \psi^{-1}(\psi(b_i)) \cap A_\fp$ for all $r+1 \leq i \leq m$. Let $A'$ be the $R_\fp$-span of the $c_i$ in $A_\fp$. Since $\psi(c_i) = b_i$ and since $( \psi(b_i))_{i=r+1}^m$ is an $R$-basis of $\prod_{t=1}^n \Mat_{n_t}(R)$ by assumption so that $(\psi(b_i))_{i=r+1}^m$ is also an $R_\fp$-basis of $\prod_{t=1}^n \Mat_{n_t}(R_\fp)$, it follows by $R_\fp$-linearity that $\psi(A') = \prod_{t=1}^n \Mat_{n_t}(R_\fp)$. Hence, $\psi(A_\fp) \sups \prod_{t=1}^n \Mat_{n_t}(R_\fp)$ and therefore
\[
\mrm{Gen}_A^\sups(\psi)  = \bigcap_{i=r+1}^m \lbrace \fp \in \Spec(R) \mid \psi^{-1}(\psi(b_i)) \cap A_\fp \neq \emptyset \rbrace \;.
\]
As each of the sets in the above finite intersection is an open generic neighborhood by Corollary \ref{br_gen_lemma_infinite}, also $\mrm{Gen}_A^\sups(\psi)$ is an open generic neighborhood.

Finally, since $\mrm{Gen}_A(\psi) = \mrm{Gen}_A^\sups(\psi) \cap \mrm{Gen}_A^\subs(\psi)$, it follows that also $\mrm{Gen}_A(\psi)$ is an open generic neighborhood.
\end{proof}

As a last ingredient we recall the following elementary lemma.

\begin{lemma} \label{algebra_split_nilp_kernel}
Suppose that $A$ is a finite-dimensional algebra over a field $K$. Then $A$ splits if and only if there exists a surjective $K$-algebra morphism $\psi:A \rarr S$ into a split semisimple $K$-algebra $S$ such that $\Ker(\psi)$ is nilpotent. The kernel of any such morphism is already equal to $\Jac(A)$ so that $A/\Jac(A) \cong S$.
\end{lemma}

\begin{proof}
If $A$ splits, then the morphism obtained by the composition $A \rarr A/\Jac(A)$ with the isomorphism $A/\Jac(A) \cong \prod_{i=1}^n \Mat_{n_i}(K)$ by the Artin--Wedderburn theorem satisfies the claimed properties since $A/\Jac(A)$ splits by \cite[7.9]{Lam-Exercises}. Conversely, assume that $\psi$ is such a morphism. Since $\Ker(\psi)$ is nilpotent, we have $\Ker(\psi) \subs \Jac(A)$. Since $\psi$ is surjective, we have $\psi(\Jac(A)) \subs \Jac(S) = 0$ and therefore $\Jac(A) \subs \Ker(\psi)$. Hence, $A/\Jac(A) = A/\Ker(\psi) \cong S$ is split and now it follows from \cite[7.9]{Lam-Noncommutative} that $A$ also splits. 
\end{proof}

\begin{theorem}  \label{split_locus_almost_open}
If $A^K$ splits, then the split locus  
\[
\mrm{SplGen}(A) = \lbrace \fp \in \Spec(R) \mid A(\fp) \tn{ splits} \rbrace
\]
is a neighborhood of the generic point in $\Spec(R)$.
\end{theorem}

\begin{proof}
We are going to show that the set $
\bigcup_{\psi} \mrm{Gen}_A(\psi)$, where $\psi$ runs over all surjective $K$-algebra morphisms $A^K \rarr S$ with nilpotent kernel into semisimple $K$-algebras, is contained in $\mrm{SplGen}(A)$. Since $A^K$ splits, such morphisms exist by Lemma \ref{algebra_split_nilp_kernel}, and since each $\mrm{Gen}_A(\psi)$ is a generic neighborhood by Proposition \ref{surjective_morphism_open_gen}, this will show that $\mrm{SplGen}(A)$ is a generic neighborhood. So, let $\psi$ be such a morphism and let $\fp \in \mrm{Gen}_A(\psi)$. This means by definition that $\psi$ restricts to a surjective morphism $\phi \dopgleich \psi|_{A_\fp}: A_\fp \rarr \prod_{i=1}^n \Mat_{n_i}(R_\fp)$. Since $\Ker(\psi) = \Jac(A^K) \unlhd A^K$ is nilpotent, also 
\[
\Ker(\phi) = A_\fp \cap \Jac(A^K) = \mrm{con}_{A_\fp}^K(\Jac(A^K))\unlhd A_\fp
\]
 is nilpotent. 
The morphism $\ol{\phi}: A(\fp) \rarr \prod_{i=1}^n \Mat_{n_i}(\rk(\fp))$ induced by $\phi$ by reducing modulo $\fp_\fp$ is actually the morphism $(\theta_{A_{\fp}}^{\fp_\fp})^*\phi$. It is clearly surjective and using \S\ref{submodule_stuff} we see that
\[
\Ker(\ol{\phi}) = \Ker((\theta_{A_{\fp}}^{\fp_\fp})^*\phi) = \mrm{ext}_{A_\fp}^{\theta_{A_\fp}^{\fp_\fp}}( \Ker(\phi)) = \mrm{ext}_{A_\fp}^{\rk(\fp)} \circ \mrm{con}_{A_\fp}^K(\Jac(A^K)) = \Jac(A^K)(\fp) \;, 
\]
i.e., $\Ker(\ol{\phi})$ is just the image of $\Ker(\phi)$ in $A(\fp)$. As $\Ker(\phi)$ is nilpotent, we thus conclude that also $\Ker(\ol{\phi})$ is nilpotent. An application of Lemma \ref{algebra_split_nilp_kernel} now shows that 
\begin{equation} \label{specialized_jac}
\Jac(A^K)(\fp) = \Jac(A(\fp))
\end{equation}
and that $A(\fp)$ splits. Hence, $\fp \in \mrm{SplGen}(A)$ and therefore $\mrm{Gen}_A(\psi) \subs \mrm{SplGen}(A)$. 
\end{proof}

\section{Generic behavior of the Jacobson radical} \label{jacobson_behavior}

Now, we come to the set $\ul{\mrm{JacDim}}\mrm{Gen}(A)$ and show that it is a neighborhood of the generic point if $A^K$ splits and the base ring is noetherian. \\

We will first consider another but similar set, namely
\[
\mrm{JacGen}(A) \dopgleich \lbrace \fp \in \Spec(R) \mid \Jac(A(\fp)) = \Jac(A^K)(\fp) \rbrace \;.
\]
 Recall that we defined and studied $\Jac(A^K)(\fp)$ already in \S\ref{tits_section}.

\begin{lemma} \label{jacgen}
If $A^K$ splits, then $\mrm{JacGen}(A)$ is a neighborhood of the generic point.
\end{lemma}

\begin{proof}
Equation (\ref{specialized_jac}) in the proof of Theorem \ref{split_locus_almost_open}  shows that if $\fp$ is contained in the generic neighborhood $\bigcup_{\psi} \mrm{Gen}_A(\psi)$, where $\psi$ runs over all surjective $K$-algebra morphisms $A^K \rarr S$ with nilpotent kernel into split semisimple $K$-algebras, then 
\[
\Jac(A(\fp)) = \Ker(\psi)(\fp) = \Jac(A^K)(\fp) \;.
\]
Hence, this generic neighborhood is contained in $\mrm{JacGen}(A)$, proving that this set is a generic neighborhood.
\end{proof}

\begin{theorem} \label{jac_dim_spec_eq_open}
If $R$ is noetherian and $A^K$ splits, then 
\[
\mrm{JacGen}(A) \cap \mrm{SplGen}(A) \subs \ul{\mrm{JacDim}}\mrm{Gen}(A) \;.
\]
In particular, $\ul{\mrm{JacDim}}\mrm{Gen}(A)$ is a neighborhood of the generic point.
\end{theorem}

\begin{proof}
Let $0 \neq \fp \in \mrm{JacGen}(A) \cap \mrm{SplGen}(A)$. Because of \S\ref{noeth_case} there is a discrete $A$-gate $(\sO,\fm)$ in $\fp$. Since $\mrm{con}_{A_\fp}^\sO \circ \mrm{con}_{A^\sO}^K \Jac(A^K) = \mrm{con}_{A_\fp}^K \Jac(A^K)$, we  have
\[
\mrm{con}_{A^\sO}^K \Jac(A^K) = \mrm{ext}_{A_\fp}^\sO \circ \mrm{con}_{A_\fp}^\sO \circ \mrm{con}_{A^\sO}^K \Jac(A^K) = \mrm{ext}_{A_\fp}^\sO \circ \mrm{con}_{A_\fp}^K \Jac(A^K)
\]
by \S\ref{submodule_stuff}. Since $\fp \in \mrm{JacGen}(A)$, we have
\[
\Jac(A^K)(\fp) = \mrm{ext}_{A_\fp}^{\rk(\fp)} \circ \mrm{con}_{A_\fp}^{\rk(\fp)} \Jac(A^K) = \Jac(A(\fp))
\] 
and we therefore get
\begin{align*}
\mrm{ext}_{A^\sO}^{\rk(\fm)} \circ \mrm{con}_{A^\sO}^K \Jac(A^K) & = \mrm{ext}_{A^\sO}^{\rk(\fm)} \circ \mrm{ext}_{A_\fp}^\sO \circ \mrm{con}_{A_\fp}^K \Jac(A^K) = \mrm{ext}_{A_\fp}^{\rk(\fm)} \circ \mrm{con}_{A_\fp}^K \Jac(A^K) \\
& = \mrm{ext}_{A(\fp)}^{\rk(\fm)} \circ \mrm{ext}_{A_\fp}^{\rk(\fp)} \circ \mrm{con}_{A_\fp}^K \Jac(A^K) = \mrm{ext}_{A(\fp)}^{\rk(\fm)} \Jac(A(\fp)) \\
& = \Jac(A^\sO(\fm)) \;.
\end{align*}
As in the proof of Theorem \ref{tits_deformation_geck} we used that $\mrm{ext}_{A(\fp)}^{\rk(\fm)} \Jac(A(\fp)) = \Jac(A^\sO(\fm))$ since $A(\fp)$ splits.  Since $A^\sO$ is $\sO$-torsion free, also $\mrm{con}_{A^\sO}^K \Jac(A^K) \leq A^\sO$ is $\sO$-torsion free and thus free since $\sO$ is a discrete valuation ring. Moreover, by \S\ref{submodule_stuff} we have
\[
\mrm{ext}_{A^\sO}^K \circ \mrm{con}_{A^\sO}^K(\Jac(A^K)) =  \Jac(A^K)
\]
and therefore $\dim_\sO \mrm{con}_{A^\sO}^K(\Jac(A^K)) = \dim_K \Jac(A^K)$. Hence,
\begin{align*}
\dim_K \Jac(A^K) & = \dim_{\sO} \mrm{con}_{A^\sO}^K \Jac(A^K) = \mrm{dim}_{\rk(\fm)} \ \mrm{ext}_{A^\sO}^{\rk(\fm)} \circ \mrm{con}_{A^\sO}^K \Jac(A^K) \\
& = \mrm{dim}_{\rk(\fm)} \Jac(A^\sO(\fm)) = \dim_{\rk(\fp)} \Jac(A(\fp)) \;.
\end{align*}
This shows that $\fp \in \ul{\mrm{JacDim}}\mrm{Gen}(A)$ and so $\ul{\mrm{JacDim}}\mrm{Gen}(A) \cap \mrm{SplGen}(A)$ is a generic neighborhood. 
\end{proof}

This theorem leads to the proof that decomposition maps are generically trivial.

\begin{proof}[{Proof} of Theorem \ref{dec_gen_trivial}]
If $R$ is noetherian and $A^K$ splits, then we have seen in Corollary \ref{decgen_cap_spl_eq_jacdimgen_cap_spl} that
\[
\mrm{DecGen}(A) \cap \mrm{SplGen}(A) = \ul{\mrm{JacDim}}\mrm{Gen}(A) \cap \mrm{SplGen}(A) \;.
\]
Theorem \ref{jac_dim_spec_eq_open} shows that $\ul{\mrm{JacDim}}\mrm{Gen}(A)$ and thus its intersection with $\mrm{SplGen}(A)$ is a neighborhood of the generic point. This intersection is contained in $\mrm{DecGen}(A)$ and so it is a neighborhood of the generic point, too.
\end{proof}

\section{Proving openness} \label{generic_representation_theory}

Now, we turn to the last part, namely proving that $\mrm{DecGen}(A)$ is open in case $R$ is noetherian and $A$ has split fibers. This turned out to be an intricate problem—even with the connection to the Jacobson radical. We solve it by recursively using the genericity of $\ul{\mrm{JacDim}}\mrm{Gen}(A)$ to show that this set is ind-constructible (in the sense of Grothendieck), and then show that it is stable under generization to deduce openness (which again relies on a theorem by Grothendieck). To make this method precise and potentially applicable to other situations we feel that it is best to formalize the notion of ``properties'' here. From these general considerations we obtain the openness of $\mrm{DecGen}(A)$ in \S\ref{lafin}.

\subsection{Specialization properties on classes of algebras}

\begin{definition}
Let $\fA$ be a class of algebras over integral domains. We write $A_{/R}$ for $A \in \fA$ with base ring $R$. Define the class
\[
\wt{\fA} \dopgleich \lbrace (A,  \fp) \mid A_{/R} \in \fA, \; \fp \in \Spec(R)  \rbrace  \;.
\]
A \word{specialization property} $\sP$ on $\fA$ is a property on the class $\wt{\fA}$. Formally, $\sP$ is a subclass of $\wt{\fA}$, and we say that $\sP(A,\fp)$ \word{holds} if $(A,\fp) \in \sP$. For $A_{/R} \in \fA$ we then define
\[
\sP\mrm{Gen}(A) \dopgleich \lbrace \fp \in \Spec(R) \mid \sP(A,\fp) \tn{ holds} \rbrace 
\]
and denote by $\sP\mrm{Ex}(A)$ the complement of $\sP\mrm{Gen}(A)$ in $\Spec(R)$.
\end{definition}

For any two specialization properties $\sP$ and $\sP'$ on $\fA$ we can form their intersection $\sP \cap \sP'$ and union $\sP \cup \sP'$.
We will essentially only consider the classes $\fF$ and $\fF_\mrm{n}$ of finite free algebras over integral domains and of finite free algebras over noetherian integral domains, respectively, and the following subclasses: the class $\fF^{\mrm{gspl}}_\star$ of such algebras having split generic fiber and the class $\fF^{\mrm{spl}}_\star$ of such algebras having split fibers. When we write $\star$ in the notation, it means that we can choose any of the classes just defined. 

\begin{example} \label{semisimplicity_property}
The following are examples of specialization properties on the class $\fF$. Here, we always denote the fraction field of the base ring of the algebra $A$ by $K$.

\begin{enum_proof}
\item $\mrm{Ss}(A,\fp)$ holds if and only if $A(\fp)$ is semisimple.

\item $\mrm{Spl}(A,\fp)$ holds if and only if $A(\fp)$ splits. 

\item $\ul{\mrm{JacDim}}(A,\fp)$ holds if and only if $\dim_K \Jac(A^K) = \dim_{\rk(\fp)} A(\fp)$.

\item $\mrm{Jac}(A,\fp)$ holds if and only if $\Jac(A(\fp)) = \Jac(A^K)(\fp)$.

\item $\mrm{Dec}(A,\fp)$ holds if and only if there is an $A$-gate in $\fp$ and $\rd_A^{\fp,\sO}:\rG_0(A^K) \rarr \rG_0(A(\fp))$ is trivial for any $A$-gate $\sO$ in $\fp$.
\end{enum_proof}

\end{example}

Note that in all the examples above the corresponding set $\sP\mrm{Gen}(A)$ coincides with the corresponding one defined before. %
We now want to study the topology of these sets in general and to do this we introduce some terminology. 

\begin{definition}
We say that a specialization property $\sP$ is \word{generic} if $(0) \in \sP\mrm{Gen}(A)$ for $A_{/R} \in \fA$ implies that $\sP\mrm{Gen}(A)$ is a \word{generic neighborhood} in $\mrm{Spec}(R)$, i.e., it contains an open neighborhood of the generic point $(0)$ of $\Spec(R)$. 
\end{definition}

A stronger version of \textit{generic} is that $\sP$ is \word{open}, meaning that $(0) \in \sP\mrm{Gen}(A)$ for $A_{/R} \in \fA$ implies that $\sP\mrm{Gen}(A)$ is open. Clearly, open properties are generic. In a similar fashion we can define \textit{closed} or \textit{constructible} properties of course. Open properties are nice since $\sP\mrm{Ex}(A)$ is then closed and can thus be described as the zero locus of an ideal $\fd_\sP(A)$ of the base ring $R$. We call this ideal the \word{$\sP$-discriminant} of $A$. What we have proven so far can be summarized as follows.

\begin{example} \label{example_generic_summary}
The following holds:
\begin{enum_proof}
\item $\mrm{Spl}$ is generic on $\fF$.
\item $\mrm{Jac}$ is generic on $\fF^{\mrm{gspl}}$.
\item $\ul{\mrm{JacDim}}$ and $\mrm{Dec}$ are generic on $\fF^{\mrm{gspl}}_{\mrm{n}}$.
\item $\ul{\mrm{JacDim}} = \mrm{Dec}$ on $\fF_{\mrm{n}}^{\mrm{spl}}$.
\end{enum_proof}
\end{example}

Cline–Parshall–\allowbreak Scott \cite{Cline:1999aa} also study several properties on classes of algebras and show that they are generic (openness is not considered). The following example from \cite{Cline:1999aa} illustrates that even ``honest algebraic'' properties might not be generic in general and so cannot be described by a discriminant as above.

\begin{example}
The semisimplicity property $\mrm{Ss}$ defined in Example \ref{semisimplicity_property} is \textit{not} generic on $\fF$. Note on the other hand, however, that if $A$ is a finite free algebra over a noetherian integral domain such that $A^K$ is split semisimple, then the semisimplicity locus of $A$ is precisely $\ul{\mrm{JacDim}}\mrm{Gen}(A)$, and this is a generic neighborhood. Hence, $\mrm{Ss}$ is in fact generic on $\fF_{\mrm{n}}^{\mrm{gspl}}$.
\end{example}

\subsection{Ind-constructible specialization properties}

Proving that a specialization property $\sP$ is generic is always the starting point. In this section we describe a condition on a generic specialization property $\sP$ which implies that $\sP$ is indeed open. The central concept in this approach is the notion of \textit{ind-constructibility} introduced in \cite[Chapitre IV, 1.9.4]{Grothendieck.A64Elements-de-geometri}. This is a weakening of constructibility and lies between generic and open. Let us recall its definition. First, a subset $E$ of a topological space $X$ is called \word{constructible} if it is a finite union of locally closed subsets of $X$, which in turn are the intersections of an open and a closed subset of $X$. It is called \word{locally constructible} if for every $x \in X$ there is an open neighborhood $U$ of $x$ in $X$ such that $E \cap U$ is constructible in $U$ (see \cite[Chapitre 0, 9.1.11]{Gro-EGA-3-1}). 

\begin{definition}
A subset $E$ of a topological space $X$ is called \word{pro-constructible} (\word{ind-constructi\-ble}) if for every $x \in X$ there is an open neighborhood $U$ of $x$ in $X$ such that $E \cap U$ is an intersection (a union) of locally constructible subsets of $U$ (see \cite[Chapitre IV, 1.9.4]{Grothendieck.A64Elements-de-geometri}). 
\end{definition}

Every intersection (union) of locally constructible subsets of $X$ is pro-constructible (ind-constructible). If $X$ is the underlying topological space of a noetherian scheme, then the ind-constructible subsets are precisely the unions of locally closed subsets (see \cite[Chapitre IV, 1.9.6]{Grothendieck.A64Elements-de-geometri}). Furthermore, if $X$ is the underlying topological space of a (not necessarily noetherian) scheme, then the ind-constructible subsets are the open sets of a topology on $X$ which is called the \word{constructible topology} and which is finer than the Zariski topology on $X$ (see \cite[Chapitre IV, 1.9.13]{Grothendieck.A64Elements-de-geometri}). Moreover, it follows from \cite[Chapitre IV, Théorème 1.10.1]{Grothendieck.A64Elements-de-geometri} that an ind-constructible subset $E$ of $X$ is open if and only if it is \word{stable under generization}, i.e., whenever $x \in X$ and $y \in  E$ with $y \in \ol{\lbrace x \rbrace}$, also $x \in E$. \\

If $\sP$ is an ind-constructible specialization property, then $\sP\mrm{Ex}(A)$ is pro-constructible and $\sP\mrm{Gen}(A)$ is open in the constructible topology on $\Spec(R)$. If $(0) \in \sP\mrm{Gen}(A)$, then it follows from \cite[Chapitre IV, 1.9.10]{Grothendieck.A64Elements-de-geometri} that $\sP\mrm{Gen}(A) \cap \ol{\lbrace (0) \rbrace} = \sP\mrm{Gen}(A)$ is a neighborhood of $(0)$ in $\Spec(R)$. Hence, ind-constructible properties are generic. The precise condition for an ind-constructible property to be open follows from what we discussed above and we record this fact once more.

\begin{theorem} \label{ind_constructible_stable_open}
An ind-constructible specialization property $\sP$ is open if and only if it is stable under generization, i.e., $(0) \in \sP\mrm{Gen}(A)$ for $A_{/R} \in \fA$ implies that if $\fq \in \sP\mrm{Gen}(A)$ and $\fp \in \Spec(R)$ with $\fq \sups \fp$, also $\fp \in \sP\mrm{Gen}(A)$. 
\end{theorem}

We are going to provide a condition ensuring that a property is ind-constructible. To this end the property has to be stable under restrictions in the following sense.

\begin{definition} \label{restriction_stable_property}
We say that a specialization property $\sP$ on $\fA$ is \word{restriction stable} if whenever $A_{/R} \in \fA$ with $(0) \in \sP\mrm{Gen}(A)$ and $\fp \in \sP\mrm{Gen}(A)$, then 
\begin{enum_thm}
\item \label{restriction_stable_property:class_stable} the $(R/\fp)$-algebra $A|_\fp \dopgleich A/\fp A$ is also contained in $\fA$,
\item $\fp/\fp \in \sP\mrm{Gen}(A|_\fp)$,
\item if $\fq/\fp \in \sP\mrm{Gen}(A|_\fp)$ for some $\fq \in \Spec(R)$ with $\fq \sups \fp$, then $\fq \in \sP\mrm{Gen}(A)$. 
\end{enum_thm}
\end{definition}

Note that the first condition is actually a stability condition on the class $\fA$ and its base rings. It certainly holds if $\fA$ itself is \word{restriction stable} meaning that condition \ref{restriction_stable_property}\ref{restriction_stable_property:class_stable} holds for all $A_{/R} \in \fA$ and $\fp \in \Spec(R)$.

\begin{example}
The classes $\fF_\star$ and $\fF^{\mrm{spl}}_\star$ are restriction stable. The class $\fF^{\mrm{gspl}}_\star$ on the other hand is \textit{not} restriction stable. For any $\fq \sups \fp$ we have a canonical isomorphism $A|_\fp(\fq/\fp) \cong A(\fq)$ of $\rk(\fq)$-algebras. This immediately implies for example that the properties $\mrm{Spl}$ and $\ul{\mrm{JacDim}}$ are restriction stable on $\fF^\star_\star$. It is also not hard to see that $\mrm{Jac}$ is restriction stable on $\fF^\star_\star$. %
\end{example}

\begin{proposition} \label{ind_constructible_property}
A restriction stable generic specialization property is ind-constructible.
\end{proposition}

\begin{proof}
Let $\sP$ be such a property on a class $\fA$ and let $A_{/R} \in \fA$. We show that $\sP\mrm{Gen}(A)$ is a union of locally closed subsets. Let $\fp \in \sP\mrm{Gen}(A)$. Note that $Z_\fp \dopgleich \Spec(R/\fp)$ is canonically homeomorphic to the closed subspace $\rV(\fp)$ in $\Spec(R)$. By assumption the generic point $\fp/\fp \in Z_\fp$ is contained in $\sP\mrm{Gen}(A|_\fp)$ and therefore $\sP\mrm{Gen}(A|_\fp)$ is a generic neighborhood in $Z_\fp$ since $\sP$ is a generic property. Let $U_\fp$ be an open neighborhood of the generic point $\fp/\fp$ of $\Spec(R/\fp)$ contained in $\sP\mrm{Gen}(A|_\fp)$. An element of $U_\fp$ can be written as $\fq/\fp$ for some $\fq \in \Spec(R)$ such that $\fq \sups \fp$. Since $U_\fp \subs \sP\mrm{Gen}(A|_\fp)$, we have $\fq/\fp \in \sP\mrm{Gen}(A|_\fp)$ and so $\fq \in \sP\mrm{Gen}(A)$ by assumption. Hence, $U_\fp$, considered as a subset of $\Spec(R)$, is contained in $\sP\mrm{Gen}(A)$. By definition of the subspace topology on $Z_\fp$, we can write $U_\fp = Z_\fp \cap V_\fp$ with an open subset $V_\fp$ of $\Spec(R)$. Hence, 
\[
\sP\mrm{Gen}(A) = \bigcup_{\fp \in \sP\mrm{Gen}(A)} U_\fp = \bigcup_{\fp \in \sP\mrm{Gen}(A)} Z_\fp \cap V_\fp \;,
\]
and this is a union of locally closed subsets. Hence $\sP\mrm{Gen}(A)$ is ind-constructible.
\end{proof}

We can now refine Example \ref{example_generic_summary} as follows.

\begin{example}
The following holds:
\begin{enum_proof}
\item $\mrm{Spl}$ is ind-constructible on $\fF$.
\item $\mrm{Jac}$ is ind-constructible on $\fF^{\mrm{gspl}}$.
\item $\ul{\mrm{JacDim}}$ is ind-constructible on $\fF^{\mrm{gspl}}_\mrm{n}$.
\item $\ul{\mrm{JacDim}} = \mrm{Dec}$ is ind-constructible on $\fF_{\mrm{n}}^{\mrm{spl}}$.
\end{enum_proof}
\end{example}

\begin{remark} \label{dec_composability}
On the class $\fF^{\mrm{spl}}_\mrm{n}$ we know that $\ul{\mrm{JacDim}} = \mrm{Dec}$ and so $\mrm{Dec}$ is restriction stable on this class. Without the connection to the Jacobson radical we would not have been able to deduce this (and for this reason we do not know about restriction stability of $\mrm{Dec}$ on larger classes than $\fF^{\mrm{spl}}_\mrm{n}$). The problem is the following. Restriction stability of $\mrm{Dec}$ on $\fF_\star^\star$ means that for any $A_{/R} \in \fF^\star_\star$ and $\fq \sups \fp$ such that $\rd_A^{\fp,\sO}:\rG_0(A^K) \rarr \rG_0(A(\fp))$ and $\rd_{A|_{\fp}}^{\fq,\sO'}:\rG_0(A(\fp)) \rarr \rG_0(A(\fq))$ exist and are trivial, also $\rd_A^{\fq,\sO''}:\rG_0(A^K) \rarr \rG_0(A(\fq))$ exists and is trivial. This seems to be difficult to prove when working directly with decomposition maps. If we would know that there are choices of gates $\sO,\sO',\sO''$ such that we can factorize 
\[
\rd_A^{\fq,\sO''} = \rd_{A|_\fp}^{\fq/\fp,\sO'} \circ \rd_A^{\fp,\sO} \;,
\]
then we would be able to deduce this. Even on $\fF^{\mrm{spl}}_\star$—where decomposition maps always exist so that this is not the central problem—such a factorization is so far only known in general by a theorem by Geck and Rouquier \cite{GR-Centers-Simple-Hecke} if both the base ring $R$ and $R/\fp$ are \textit{normal}—the latter being a quite restrictive assumption not allowing us to deduce this on $\fF^{\mrm{spl}}_\star$ in this way. 
\end{remark}

\subsection{Stratification defined by an open specialization property}

In the following we say that a specialization property $\sP$ on $\fA$ is \word{generically true} if $\sP(A,0)$ holds for all $A_{/R} \in \fA$, where $(0)$ is the generic point of $\Spec(R)$. This notion is adapted to properties like $\ul{\mrm{JacDim}}$ which compare the generic fiber with a special fiber (and are thus obviously generically true), and less useful for properties like $\mrm{Spl}$ which just consider a special fiber. In the following lemma we show that in the noetherian case we only have to study a property in finitely many fibers to know it for all fibers.

\begin{lemma} \label{stratification_lemma}
Let $\sP$ be an open and generically true specialization property on a restriction stable class $\fA$. If $A_{/R} \in \fA$ with $R$ noetherian, then there is a finite number of points $\xi_1,\ldots,\xi_n$ in $X \dopgleich \Spec(R)$ such that
\[
X = \bigcup_{i=1}^n \sP\mrm{Gen}(A|_{\xi_i}) \;.
\]
The subset $ \sP\mrm{Gen}(A|_{\xi_i})$ is locally closed in $X$ and $\xi_i$ is the generic point of its closure.
\end{lemma}

\begin{proof}
We start with the generic point $\xi_1 \dopgleich (0)$ of $X$. Since $\sP$ is open and generically true, $X_1 \dopgleich \sP\mrm{Gen}(A)$ is open in $X$. Let $Z_1$ be the complement of $X_1$. This set is closed and we have $X = X_1 \cup Z_1$. Now, we continue with $Z_1$. Let $\xi_2,\ldots,\xi_{n_2}$ be the generic points of the closed complement $Z_1$ of $X_1$. As $\fA$ is restriction stable, the restrictions $A|_{\xi_i}$ are contained in $\fA$. Since $\sP$ is generically true, each set $X_{2,i} \dopgleich \sP\mrm{Gen}(A|_{\xi_i})$ is open in $Z_1$, and thus locally closed in $X$. The complement $Z_{2,i}$ of $X_{2,i}$ in $Z_1$ is closed in $Z_1$, thus closed in $X$, and we can write $Z_1 = \bigcup_{i=2}^{n_2} X_{2,i} \cup \bigcup_{i=2}^{n_2} Z_{2,i}$. We can now continue in the same way with the sets $Z_{2,i}$. As $X$ is noetherian, this process will eventually terminate.
\end{proof}

Clearly, if we always join all the sets $\sP\mrm{Gen}(A|_{\xi_i})$ in Lemma \ref{stratification_lemma} which have non-trivial intersection, we can write $X$ as a disjoint union, and this can be refined to an actual stratification.

\subsection{Monotone specialization invariants}

We finish our formal treatment of properties of algebras by a construction producing open and generically true properties.

\begin{lemma} \label{monotone_invariant}
A \word{monotone specialization invariant} on a restriction stable class $\fA$ of algebras over integral domains is a map $\sI: \wt{\fA} \rarr (S,\leq)$ to a poset such that the following holds for all $A_{/R} \in \fA$:
\begin{enum_thm}
\item $\sI(A,\fp) \leq \sI(A,0)$ for all $\fp \in \Spec(R)$,
\item $\sI(A|_\fp, \fq/\fp) = \sI(A,\fq)$ for all $\fq,\fp \in \Spec(R)$ with $\fq \sups \fp$.
\end{enum_thm}
If for such an invariant the \word{associated specialization property} $\ul{\sI}$ on $\fA$ with
\[
\ul{\sI}(A,\fp) \tn{ holds if and only if } \sI(A,\fp) = \sI(A,0)
\]
is generic, then it is already open and generically true. If in this case $A_{/R} \in \fA$ with $R$ noetherian, then $\Spec(R)$ can be stratified into finitely many strata such that $\sI$ is constant on each of them.
\end{lemma}

\begin{proof}
It is obvious that $\ul{\sI}$ is stable under restrictions. If $\fq \sups \fp$, then $\fq/\fp \sups \fp/\fp$ and so by assumption 
\[
\sI(A,\fq) = \sI(A|_\fp,\fq/\fp) \leq \sI(A|_\fp,\fp/\fp) = \sI(A,\fp) \leq \sI(A,0) \;.
\]
Hence, if $(0),\fq \in \ul{\sI}\mrm{Gen}(A)$, then $\sI(A,\fq) = \sI(A,0)$ and the inequalities above imply that $\sI(A,\fp) = \sI(A,0)$, so $\fp \in \ul{\sI}\mrm{Gen}(A)$. This shows that $\ul{\sI}$ is stable under generization and so Theorem \ref{ind_constructible_stable_open} implies that $\ul{\sI}$ is open. Finally, it is clear that $\ul{\sI}$ is generically true and so the claim about the stratification follows immediately from Lemma \ref{stratification_lemma}.
\end{proof}

\subsection{La fin—application to decomposition maps} \label{lafin}

We will now come to the end of this paper and prove that $\mrm{DecGen}(A)$ is open whenever $A$ has split fibers and its base ring is noetherian. To do this, we study the specialization invariant $\mrm{JacDim}: \wt{\fF} \rarr (\bbN,\geq )$ mapping $(A,\fp)$ to $\dim_{\rk(\fp)} \Jac(A(\fp))$. The property $\ul{\mrm{JacGen}}$ we studied in the paragraphs above is clearly the property associated to this invariant in the sense of Lemma \ref{monotone_invariant}. Our aim is to show that $\mrm{JacDim}$ is monotone on $\fF_{\mrm{n}}^{\mrm{spl}}$.

\begin{theorem} \label{jac_dim_spec_geq_general}
Suppose that $R$ is noetherian and let $\fp \in \Spec(R)$ such that $A(\fp)$ splits. Then 
\[
\dim_{\rk(\fp)} \Jac(A(\fp)) \geq \dim_K \Jac(A^K) \;.
\]
\end{theorem}

\begin{proof}
Because of \S\ref{noeth_case} there is a discrete $A$-gate $(\sO,\fm)$ in $\fp$. Let $J \dopgleich \mrm{con}_{A^\sO}^K(\Jac(A^K)) = A^\sO \cap \Jac(A^K)$.  According to \S\ref{submodule_stuff} the quotient $C \dopgleich A^\sO/J$ is $\sO$-torsion free and thus already $\sO$-free since $\sO$ is a discrete valuation ring. Since $A^\sO$ is $\sO$-torsion free, also $J$ is already $\sO$-free. Moreover, by \S\ref{submodule_stuff} we have
\[
\mrm{ext}_{A^\sO}^K(J) = \mrm{ext}_{A^\sO}^K \circ \mrm{con}_{A^\sO}^K(\Jac(A^K)) = \Jac(A^K)
\]
and therefore $\dim_\sO J = \dim_K \Jac(A^K)$. Note that
\[
 \mrm{ext}_{A^\sO}^{\rk(\fm)}(J)  = \mrm{ext}_{A^\sO}^{\rk(\fm)} \circ \mrm{con}_{A^\sO}^{\rk(\fm)} ( \Jac(A^K)) = \Jac(A^K)(\fm)  \subs \Jac(A^\sO(\fm))
 \]
 by Lemma \ref{jac_spec_dim_lemma} so that
 \[
 \dim_{\rk(\fm)}  \mrm{ext}_{A^\sO}^{\rk(\fm)}(J) \leq \dim_{\rk(\fm)}  \Jac(A^\sO(\fm)) \;.
 \] 
  By \S\ref{submodule_stuff} we have
\[
C^{\rk(\fm)} \cong (A^{\sO})^{\rk(\fm)}/\mrm{ext}_{A^{\sO}}^{\rk(\fm)}(J) = A^{\sO}(\fm)/\mrm{ext}_{A^{\sO}}^{\rk(\fm)}(J) 
\]
and therefore
\begin{align*}
\dim_\sO A^\sO - \dim_\sO J & = \dim_\sO C = \dim_{\rk(\fm)} C^{\rk(\fm)} \\
& = \dim_{\rk(\fm)} A^{\sO}(\fm) - \dim_{\rk(\fm)} \mrm{ext}_{A^{\sO}}^{\rk(\fm)}(J) \;,
\end{align*}
implying that 
\[
\dim_K \Jac(A^K) = \dim_\sO J = \dim_{\rk(\fm)} \mrm{ext}_{A^{\sO}}^{\rk(\fm)}(J) \leq \dim_{\rk(\fm)}  \Jac(A^\sO(\fm))\;.
\]
But since $A(\fp)$ is splits we have 
\[
\dim_{\rk(\fp)} \Jac(A(\fp)) = \dim_{\rk(\fm)} \Jac(A^{\sO}(\fm)) \geq 
\dim_K \Jac(A^K)  \;.
\]
\end{proof}

The preceding theorem in combination with Lemma \ref{monotone_invariant} immediately implies the following corollary, which in turn proves Theorem \ref{dec_stratification}.

\begin{corollary}
$\mrm{JacDim}$ is a monotone specialization invariant on $\fF_{\mrm{n}}^{\mrm{spl}}$. In particular, if $R$ is noetherian and $A$ has split fibers, then $\ul{\mrm{JacDim}}\mrm{Gen}(A) = \mrm{DecGen}(A)$ is open. 
\end{corollary}

\section{Open questions} \label{questions}

We finish with a list of open questions we think are worth investigating. 

\begin{enumerate}
\item How can we explicitly describe the decomposition discriminant $\fd_A$? Can we extract from it a ``generalized Schur element'' for each simple module, which repairs those for which the actual Schur element is zero and thus does not contain any information?

\item What are further geometric properties of $\mrm{DecEx}(A)$? Is it either empty or pure of codimension one (if not in general, in which general situations is this true)? This is related to Remark \ref{codim1}.

\item What can be said without assuming that the base ring $R$ is noetherian? Is it still true that $\ul{\mrm{JacGen}}(A) = \mrm{DecGen}(A)$? Is it still true that $\ul{\mrm{JacGen}}(A) $ is open if $A$ has split fibers?

\item Similarly to the last question it would be important to know if we can remove splitting assumptions from our results. This is probably difficult and does not lead to nice results as the non-genericity of the semisimplicity property shows in the non-split case. On the other hand, condition \ref{gate_definition_G} in Definition \ref{gate_definition} is weaker than the special fiber being split—this is just the easiest situation where it holds. In which general situations do we know that this condition holds?

\item If $A^K$ splits, is it true that the split locus $\mrm{SplGen}(A)$ is not only ind-constructible but even open? To us it is not clear why the splitting property should be stable under generization. 

\item It would be very helpful (also in applications) to have a general composability result of decomposition maps as discussed in Remark \ref{dec_composability}. For this one might have to pick an appropriate choice of $A$-gates. In principle one might be able to deduce ind-constructibility and stability under generization of $\mrm{DecGen}(A)$ in this way.

\item We have proven that triviality of a decomposition map does not depend on the choice of the $A$-gate. What is the general relationship between decomposition maps for different $A$-gates in a fixed prime? Is there some sort of a Galois automorphism interchanging the decomposition matrices as it is in the case of finite groups?

\item In general it would be interesting to have counter-examples for everything which does not hold in general. These seem to be hard to find as all well-understood cases (e.g., group algebras or Hecke algebras) are too well-behaved.
\end{enumerate}

\bibliography{references}{}
\bibliographystyle{plain}

\end{document}